\documentclass{article}
\usepackage[utf8]{inputenc}
\usepackage[english]{babel}
\usepackage[T1]{fontenc}
\usepackage{lmodern}

\usepackage{a4,amsmath,amssymb,amsthm, amscd}
\usepackage{tikz}
\usepackage{subfig}
\usepackage{hyperref}
\usepackage[all]{xy}
\usepackage{csquotes}
\usepackage[capitalize,nosort]{cleveref}
\newcommand{\tcb}{\textcolor{blue}}

\title{On balance properties of  hypercubic  billiard   words}

 \author{N. Bédaride\footnote{
    Aix Marseille Université, CNRS,
    Centrale Marseille, I2M, UMR 7373, 13453 Marseille, France.
    Email: \texttt{nicolas.bedaride@univ-amu.fr}}\thanks{ This work was supported by the Agence Nationale de la Recherche  through the project  ``IZES''  (ANR-22-CE40-0011).} 
 \and
 V.~Berth\'e\footnote{Universit\'e  Paris Cit\'e, IRIF, CNRS, 
F-75013 Paris, France. Email: \texttt{berthe@irif.fr}}\thanks{ This work was supported by the Agence Nationale de la Recherche  through the project  ``SymDynAr''  (ANR-23-CE40-0024-01).}
 \and
A. Julien\footnote{Nord University,     Faculty of teacher education, Levanger, Norway. Email:
\texttt{antoine.julien@nord.no}}}

\newtheorem{lemma}{Lemma}
\newtheorem{theorem}[lemma]{Theorem}

\newtheorem{proposition}[lemma]{Proposition}
\newtheorem{corollary}[lemma]{Corollary}
\newtheorem{definition}[lemma]{Definition}
\theoremstyle{definition}

\newtheorem{remark}[lemma]{Remark}
\newtheorem{example}[lemma]{Example}

\newcommand{\T}{\mathbb{T}}
\newcommand{\R}{\mathbb{R}}
\newcommand{\Z}{\mathbb{Z}}
\newcommand{\cH}{\mathcal{H}}
\newcommand{\N}{\mathbb{N}}

\newcommand{\A}{\mathcal{A}}
\newcommand{\bA}{\mathbb{A}}

\newcommand{\C}{\mathcal{C}}
\newcommand{\btheta}{\boldsymbol{\theta}}
\newcommand{\balpha}{\boldsymbol{\alpha}}

\DeclareMathOperator{\Card}{Card}

\date{\today}

\begin{document}

\maketitle

\begin{abstract}
This  paper studies balance properties for billiard words.  Billiard words generalize Sturmian words by coding trajectories in hypercubic billiards.
In the setting of  aperiodic order,   they also provide the simplest examples of quasicrystals, as  tilings of the line  obtained  via      cut and project sets with a cubical canonical window.  
  By construction, the number of occurrences of each letter in a  factor  (i.e., a string of consecutive letters) of a hypercubic billiard word  only depends on the length of the factor,  up to an additive constant.  In other words, the difference of the number of occurrences of each letter in factors  of the same length is bounded.
 In contrast with the behaviour of letters, we prove the existence of words  that are not balanced in   billiard words:  the difference of the number of occurrences of such  unbalanced   factors in  longer factors  of the same length is unbounded.  The proof relies  both on topological methods inspired by tiling  cohomology and  on  arithmetic results on bounded remainder sets for toral translations.
\end{abstract}


\section{Introduction}
Given  an infinite word  $u$ over  a finite alphabet,  the function  which counts  the number of factors  of a given  length of $u$  (i.e.,  the  strings of consecutive letters that occur in $u$)    is a classical measure of disorder, which 
extends naturally to  tilings, by counting patterns  of a given  size, and to   higher-dimensional words, by  counting  e.g. the number of  rectangular factors of a  given size. 
This  has led to  striking   results, in particular in the direction of Nivat's conjecture,  which  aims to extend    the  following simple characterization of periodicity:  if  an infinite   word   over a finite alphabet     has at most $n$  factors of length $n$, then it 
is ultimately  periodic. This result, which   is known as  Morse and Hedlund's theorem \cite{MorseHedlund:40},  admits  a  simple and combinatorial proof. 
Nivat's conjecture   then states  that if a two-dimensional word  admits  at most $ m n$ rectangular configurations of size $(m,n)$, then it admits at least one direction of periodicity. There have been exciting recent  developments concerning   this  conjecture and  despite its apparent simplicity, one of the most promising approach  involves non-trivial  algebraic methods.  
The algebraic  approach developed by J. Kari and his coauthors  in  \cite{Sza:18,KariMoutot:19,KariSza:2020,KariMoutot:20,KaMou:23}  is indeed   based on the  representation of   configurations   as formal power series;   periodicity for a  configuration  is then  expressed  as   having  a   non-trivial annihilator which, through   Hilbert's  Nullstellensatz,
yields  powerful    decomposition results.       Hence,
low complexity configurations can be decomposed into a sum of finitely many periodic configurations, but with possibly infinite alphabet 
(which also yields  an  asymptotic version of   Nivat's conjecture) \cite{KariSza:2020}.

 We focus here on the  notion of balance which can  also  be considered  as a measure of disorder. This is a priori a purely combinatorial notion, and yet,
  we will see here also that the involved methods  go far beyond combinatorics. 
An infinite  $u$ over  a finite alphabet  is said to be \emph{balanced}   if  letters   occur evenly  in    the  factors  of  $u$ of the same  length. More precisely,
there exists a constant~$C$ such that for every letter  $a$ and for every pair $w,w'$ of factors of~$u$ of the same length, 
the number of occurrences of $a$ in $w$ differs from its number of occurrences in $w$ by at most $C$.  In other words, the number of occurrences of each letter only depends on the length of the word, up to an additive constant. This  notion extends to  factors by replacing occurrences of letters  by occurrences of factors.

For more on the subject, see the survey \cite{Vuillon:03}. 

Balance was first studied by Morse and Hedlund  in the form of $1$-balance (i.e., with $C=1$) for letters   in the setting of words over  a two-letter alphabet in the seminal papers \cite{MorseHedlund:38,MorseHedlund:40}; they  prove that the family   of $1$-balanced aperiodic infinite words over a two-letter alphabet  coincides with the family of   Sturmian words.    Moreover, Sturmian words   have been proved to be  balanced on their factors in \cite{FagnotVuillon:02}.
Words over a larger alphabet that  are $1$-balanced have been characterized in \cite{Hubert:00} (see also \cite{Graham:73}) and shown to be closely related to Sturmian words.

 In the multidimensional framework, balance has been considered  both for    multidimensional words
 \cite{BerTij:2002} and  for tilings  \cite{Sadun:16}.   
The  study of balance belongs to the  general domain of  aperiodic order  which    refers to the mathematical   formalization of quasicrystals (see e.g. \cite{BaakeGrimm}). For tilings and  Delone sets, balance is   closely related to the notion of   bounded distance equivalence to  lattice   \cite{Lacz:92,HayKeKo:17}. 
More generally, the notion of balance  occurs in number theory  in  connection with  Fraenkel's  conjecture \cite{Fraenkel:73,Tijdeman:2000},  in operations research, for optimal routing  and scheduling (see e.g.  \cite{AltmanGH:00,BraunerCrama:04,BraunerJost:08,Tijdeman:80}). Balance  is   also  closely related to  symbolic discrepancy, as   
investigated in \cite{Adam:03,Adam:04}.  

There is a classical way of constructing balanced sequences, and in particular Sturmian sequences, in terms of cutting sequences. 
 Sturmian sequences are  indeed codings of trajectories of billiards on a square table with an irrational direction and, by unfolding trajectories, they code whether a horizontal or a vertical side of a lattice square is hit. Billiard words  generalize Sturmian words 
 to larger alphabets by considering  hypercubic  billiards, as   studied in  \cite{Ar.Ma.Sh.Ta.94,Ar.Ma.Sh.Ta.94bis}.
Billiard words can equivalently be defined  in terms of  codings of toral translations. 
Indeed,   billiard words   on $ d+1$ letters code   translations of  the   torus $\T^d$.  Balance properties for letters  for billiard words have   been  considered in \cite{Andrieu,Andrieu-Vivion-23}; see also  \cite{BCCSY} which studies
 generalizations of billiard words with low discrepancy.

The aim of this paper is to study balance properties for factors of   hypercubic billiard words  by using two types of results.
We use on the one hand    topological properties concerning  the cohomology of tilings for cut and project sets  (see \cref{sec:cohobi}),  and, on the other hand,    results developed  in \cite{GL:15,GrepLar:16},  whose main concerns are bounded remainder sets for toral translations, with methods relying  on  Fourier analysis.  

We now state the main results of this paper.
The precise definitions concerning the subshift $X_{\btheta}$  discussed  below are given in \cref{sec:defbilliard}.

\begin{theorem}\label{theo:main}
Let $d \geq 2$.
Let $\btheta= (1,\theta_1, \ldots,\theta_d)\in {\mathbb R}^{d+1}_+$,   and let $(X_{\btheta},S)$ be the subshift defined by the  hypercubic billiard   in $\mathbb R^{d+1}$ with  direction $\btheta$. If $X_{\btheta}$ is minimal, then its language is not balanced on factors. 

\end{theorem}

In the case of $d=2$, that is, for the cubic billiard case, we can be more specific. 
\begin{proposition}\label{prop:long2}
Let $\btheta= (1,\theta_1, \theta_2)\in {\mathbb R}^{3}_+$, with $1 ,\theta_1, \theta_2$ being rationally independent.
Let $(X_{\btheta}, S)$ be  the associated   billiard subshift. 
Then no factor of $X_{\btheta}$ of length  larger than or  equal  to two is  balanced.
\end{proposition}

As a direct  corollary of \cref{theo:main}, we get that     a minimal  billiard shift $(X_{\btheta}, S)$ in ${\mathbb R}^{d+1}$ cannot be a   primitive unimodular proper $S$-adic subshift on a $d+1$-letter alphabet by \cite[Corollary 5.5]{BCDLPP:19}. Indeed,   for  primitive unimodular proper $S$-adic subshift on a $d+1$-letter alphabet, balance on letters is equivalent with balance on factors.  Here we prove that  minimal billiard subshifts are  only balanced on letters.

We briefly describe the contents of this paper. Basic definitions are recalled in  \cref{sec:basicnotions}. \cref{sec:defbilliard}
introduces cubic billiard words. In particular, 
\cref{sec:CP}  describes billiard words in terms of cut and project sets. 
\cref{sec:coho} recalls   cohomological definitions and \cref{sec:cohobi} details cohomological properties of billiard words. The proof of 
\cref{theo:main} is given in \cref{subsec:proofmain}, and the proof of  \cref{prop:long2} is given in 
\cref{subsec:proof2}.

\section{Basic notions} \label{sec:basicnotions}
\subsection{Words, subshifts and dynamical systems} 
For the definitions of this paragraph, we refer to \cite{Pyth.02,Queffelec:2010}.
Let $\mathcal{A}$ be a finite set, called the {\em alphabet}. We assume that its cardinality $d$ satisfies $d\ge 2$. Elements of $\mathcal{A}$ are called \emph{letters}.
A word is a finite   string of letters. If $w$ is the  finite word $w=w_0\dots w_{n-1}$, then $n$ is called the  {\em length} of the word $w$ and is denoted by $|w|$. The set of all finite words over $\mathcal{A}$ is denoted by $\mathcal{A}^*$. 
If $w=w_{0}\cdots w_{n-1}$ is a word, a prefix of $w$ is any factor $w_{0}\cdots w_{j}$ with $j\le n-1$. A suffix of $w$ is any word of the form $w_{j}\cdots w_{n-1}$ with $0\le j\le n-1$.

The set   $\mathcal{A}^\mathbb{Z}$  stands for the set of biinfinite words   with values in the alphabet  $\mathcal{A} $.
A word $w=w_0\dots w_{r-1}$ is said to occur at position $m \in \Z$ in the biinfinite word $x \in \mathcal{A}^\mathbb{Z} $ if  $x_{m+i}=w_i$ for all  integers $i\in[0,r-1]$. We say that the word $w$ is a \emph{factor} of $x$. We  use the notation $x_{[k,\ell]}$ for the  factor $x_k \cdots x_{\ell}$ of $x$, with $k, \ell \in \mathbb{Z}$. The language $\mathcal L(x)$ of $x$   is the set of all words  in $\mathcal{A}^*$ which occur in $x$, and the language   $\mathcal L_{n}(x)$ of length $n$ of $x$  is the set of all  factors of $x$ of length $n$.   Let $x \in \mathcal{A}^{\mathbb{Z}}$ and let $v \in \mathcal{A}^*$.  For $w \in  \mathcal{A}^*$,  the notation $|w|_v$ stands for the number of occurrences of the word $v$ in  $w$.

We recall that $d \geq 2$ stands for  the  cardinality of $\mathcal A$.   We endow $\mathcal A$ with the discrete topology and consider the product topology on $\mathcal A^{\mathbb Z}$. This topology is compatible with the distance $d$ on $\mathcal A^{\mathbb Z}$ defined for $x\neq y$  by
$$d(x,y)=\frac{1}{d^n}\quad \text{if}\quad n=\min\{|i|\geq 0 \,: \,  x_i\neq y_i\}.$$

The shift map $S$ is the map defined on    $\mathcal{A}^\mathbb{Z}$  by $S(x)=y$ with $y_n=x_{n+1}$ for each integer $n$.

\begin{definition}\label{def-rec-lin}
The  symbolic dynamical system  associated with the  biinfinite word $x$ is  defined as the symbolic dynamical system $(X_x, S)$, where $S$ is the shift map and $X_x=\overline{\{S^n(x) \,: \,  n\in\mathbb{Z}\}}$.

\end{definition}

We have the  equivalence between $y \in X_x$ and $\mathcal L(y) \subset \mathcal L(x)$. Thus the language of $X_x$ is equal to the language of $x$. 
A  {\em subshift}  $(X,S)$   (also called shift for short) is   a closed shift-invariant of  $\mathcal{A}^\mathbb{Z}$. 
It is said  {\em minimal} if  its closed shift-invariant  subsets are trivial. In that case, all the words in $X$ have the same language.
In particular, a  symbolic dynamical system of the form  $(X_x, S)$ is a subshift.  The \emph{factor complexity}~$p_X(n)$ of a minimal shift $(X,S)$ is defined  as 
$p_X(n) = {\Card} \mathcal{L}_n(X)$.

 Let $(X,S)$ be a subshift.
Given a  word $w=w_0 \cdots w_{n-1}  \in  \mathcal{A}^*$, the  {\em prefix cylinder} $[w]$ is defined as  the set of biinfinite words  in $X$  that  have $w$ as a prefix, i.e.,
$$[w]=\{ x \in  X \,: \,  x_{[0,n-1]}= w\},$$
and for a  word $w=w_{-m+1}  \cdots w_{n-1}  \in  \mathcal{A}^*$, the {\em two-sided  cylinder} $[[w]]$ is  defined as 
$$[[w]]=\{ x \in  X \,: \,  x_{[-m,n]}= w\}.$$

 We will use the notation ${\mathbf 1}_w$       for the {\em characteristic function}  of the  prefix cylinder $[w]$  on $X$, i.e., 
 ${\mathbf 1}_w \colon X \rightarrow \{0,1\}$,  $x \mapsto 1$ if $x_{0}\cdots x_{|w-1|}= w$,  and  $x \mapsto 0$, otherwise.
 We similarly define the characteristic function  associated  with  a two-sided cylinder $[[w]]$.

 Let $(X,S)$ be a subshift with $X \subset \mathcal{A}^{\mathbb{Z}}$. 
A~probability measure~$\mu$ on~$X$ is said to be \emph{$S$-invariant} if $\mu(S^{-1} B) = \mu(B)$ for every Borel set $B \subset X$.
  An invariant probability measure  $\mu$ on $X$ is  \emph{ergodic} if any shift-invariant measurable set
has either measure $0$ or $1$.
The shift $(X,S)$ is \emph{uniquely ergodic} if there exists a unique shift-invariant probability measure on~$X$. 
In all that follows,  for a  uniquely ergodic  subshift $(X,S)$,   the notation $(X,S,\mu)$   means that $\mu$ is the  unique shift-invariant probability measure on~$X$

Let $f$ be a function  defined on the subshift  $(X,S)$ with values in ${\mathbb R}$. We will use the notation  $S_nf$ for the  {\em Birkhoff sum} defined for any $x \in X$ and $n \in {\mathbb N}$ as 
  \[
S_n f(x):=   f(x) + f\circ S(x) + \cdots + f \circ  S^{n-1}(x).
 \] Birkhoff sums are central elements  in  ergodic theory. 
 Indeed, unique ergodicity  allows one to  relate spatial means to temporal means  via  {\em Birkhoff ergodic
 theorem}, i.e., 
 for any  continuous function $ f\in  \C(X,\R)$, one has 
\begin{equation} \label{eq:birk}
\frac{1}{n} S_n f(x)=\frac{1}{n} \sum_{k=0}^{n-1}
f \circ S^k \xrightarrow[ n \rightarrow \infty]{}
\int_{X} f\, d\mu  \mbox{ for all } x \in X .
\end{equation}
Let   now  assume that $(X,S)$ is a minimal and  uniquely ergodic subshift with shift-invariant  probability measure $\mu$. 
Then,   by taking  in \cref{eq:birk}  the characteristic function of a two-sided cylinder $[[v]]$, with $v \in {\mathcal L}(X)$, one gets   the existence of the following limit  and its relation with the measure $\mu$:
 $$\lim_{n \to +\infty} \big|x_{[-n,n]}|_v/2n = \mu[v] \mbox{ for every }x \in X.$$  We  also call  this limit the \emph{frequency} of~$v$ in~$x$. 
Also  one has  $\lim_{n \to +\infty} \big|x_{[0,n-1]}|_v/n = \mu[[v]]$  for every $x \in X.$
For more on  factor frequencies  and invariant measure, see \cite[Chapter~7]{CANT}.

We will  work both with   subshifts and dynamical  systems of a more geometric nature. Classical  geometric dynamical systems under consideration in this paper are    toral translations of the form  $T_{\balpha}$, with  $\balpha= (\alpha_1,\dots,\alpha_{d}) $  being  a vector in $[0,1]^d$,   defined on the $d$-dimensional torus  $\mathbb{T}^{d} = \mathbb{R}^{d}/\mathbb{Z}^{d}$ by \[
R_{\balpha}:\  \mathbb{T}^{d} \rightarrow \mathbb{T}^{d}, \quad  {\mathbf x} \mapsto  {\mathbf x}  + (\alpha_1,\dots,\alpha_{d})  \pmod{\mathbb{Z}^{d}}.
\]

By a {\em topological dynamical system},
we mean a pair $(X, T )$ where $X$ is a compact metric space and  $ T \colon  X \rightarrow  X$
is a homeomorphism. Two topological  dynamical systems $(X_1,T_1)$ and $(X_2,T_2)$ are said to  be  {\em conjugate} if there is  a  conjugacy between them,
i.e., a homeomorphism    $ \psi: X_1 \rightarrow X_2$, such that
$ \psi \circ  T_1= T_2\circ \psi$, that is, $$
\begin{CD}
X_{1} @>T_1 >> X_{1}\\
@VV\psi V @VV\psi V\\
X _{2}@>T_2>> X_2
\end{CD}
$$

For topological reasons, one cannot have  homeomorphisms between  symbolic and geometric  dynamical systems.
In order to  dynamically convey the idea of  a subshift coding a geometric  dynamical system, we need the following notion of a conjugacy  up to sets of zero measure.
A measurable dynamical system  $(X,T, \mu)$ is such that 
$T$ is a measurable map and $\mu$ is a $T$-invariant probability measure  on $X$,  i.e.,  $\mu(T^{-1}(B) )=\mu(B)$ for every measurable set $B$ and $\mu(X)=1$.
Two dynamical systems $(X_1,T_1,\mu_1)$ and $(X_2,T_2,\mu_2)$ are \emph{measurably}  isomorphic if there are measurable subsets $Y_1 \subset  X_1$, $Y_2 \subset  X_2$
of measure $1$ and a bijection $\psi \colon Y_1 \rightarrow Y_2$ with  $\psi$ and $\psi^{-1}$ being measurable, such that 
$\psi \circ T_1(x)= T_2 \circ \psi(x)$ for every $x \in Y_1$, and $\mu_2 \circ \psi(B)= \mu_1(B)$ for every measurable set $B \subset Y_1$.

 \subsection{Balance and discrepancy}\label{subsec:balance}
There are some factors for  which  the convergence in \eqref{eq:birk} can be significantly improved. This is the situation discussed in this section in terms
 of balance and bounded discrepancy.   
A biinfinite word $x$ is said to be {\em  balanced} on the factor $v$ if there exists a constant $C>0$ such that for every pair $(u,u')$ of factors of $x$ of  the same length we have $$| |u|_v-|u'|_v|\leq C.$$

A biinfinite word $x \in  {\mathcal A}^{\Z}$ is {\em balanced on letters} if it is balanced on each letter in  $\mathcal A$. It is  {\em balanced on factors of length $n$} if it is balanced on all its factors of length $n$, and, lastly, it is {\em balanced on factors} if it is balanced on all its factors. 
Let $(X,S)$ be a minimal subshift. Similarly, it is said to be  balanced on letters  if   there exists  $x\in X$  which is  balanced on  letters,
 it is balanced on factors of length $n$ if  there exists  $x\in X$ which is  balanced on   factors of length $n$, and  it is balanced on factors if it is balanced on all its factors.
By minimality, since all  the elements of a minimal subshift have the same language, one   can replace above  ``there exists $x \in X$''  by   ``for all $x \in X$''.
By \cite[Lemma 23]{Adam:03},  if   a minimal subshift $(X,S)$ is balanced for length $n$, then it is balanced for all lengths $k$ which are smaller than or equal to  $  n$.

Let  $x \in \mathcal{A}^{\mathbb{Z}}$, and let $ w \in {\mathcal L} (x)$. One has 
$$S_n({\mathbf 1}_w)=\Card \bigl\{ i \in [0, n-1]  \,: \,  x_i \cdots x_{i+|w|-1}= w \bigr\} =  |x_{[0,n-1]}|_w .$$
The quantity  $ \sup_n |S_n({\mathbf 1}_w) - n \mu[w]|$ is called the {\em discrepancy} of $x$ with respect to $w$. Observe that $S_n({\mathbf 1}_w) - n \mu[w]=S_n({\mathbf 1}_w- {\mathbf 1})$ where ${\mathbf1}$ is the constant function.

The following result is a folklore result (see e.g. \cite{BerTij:2002}).  We recall its proof for the sake of self-containedness.

\begin{proposition}\label{prop:disc}
Let $(X,S)$ be a minimal subshift and let $x \in X$. Let $ w \in {\mathcal L} (X)$.
 The word $x$ is balanced on $w$ if and only if there exist $C\geq 0$ and $  \mu_w  \geq 0$ such that 
 $|S_n({\mathbf 1}_w) - n \mu_w| \leq C$ for all $n$.
\end{proposition}
\begin{proof}

Let $C>0$ be  such that $x$ is $C$-balanced  on  $w$. 
For every non-negative integer $p$, let $N_p$ be    such that
for every factor $v$  of length $p$ of $x$,  $|v|_w$ 
 belongs to the set $\{N,N+1,\cdots, N+C\}$.
 
We first observe that the sequence $(N_p/p)_{p \in {\mathbb N^*}}$ is a Cauchy sequence.
Indeed, set $v = x_0 \cdots x_{pq-1}$, with $p,q$ positive integers.
One has
\[
p N_q \leq |v|_w \leq p N_q + p C, \ \quad q N_p \leq |v|_w \leq q N_p + q C.
\]
We deduce that $-q C \leq q N_p -p N_q \leq p C$ and thus $-C \leq N_p - p N_q /q \leq p C / q$.
 
Let $\mu_w$ stand for $\lim_q N_q/q$. By letting $q$ tend to infinity, one then deduces that 
$-C \leq N_p -p \mu_w \leq 0$. Observe that $\mu_w$ is the frequency of $w$ in $x$. Indeed, by minimality, this is sufficient to consider   the non-negative one-sided part of $x$,
i.e., $(x_n)_{n \geq N}$.

The other direction   comes from triangular inequality.
\end{proof}
We now give  a geometric interpretation of balance on letters.

 \begin{lemma}\label{lem:dist}
 Let $(X,S)$, with $X \subset {\mathcal A}^{\mathbb Z}$ be a minimal subshift, and let $x \in X$. Let $d=\Card {\mathcal A}$.
 The  biinfinite word $x$ is  balanced  on letters if and only if there exists  a line in $\mathbb R^d$ such that   the discrete line with  set of  vertices
 $\{(|x_{[-n,0]}|_a)_{a \in \mathcal A} \,;\,  n \in \N\} \cup  \{( |x_{[0,n]}|_a)_{a \in \mathcal A}  \,;\,  n \in \N\} $, which represents geometrically  $x$, stays  within   bounded distance from this line.
 \end{lemma}
 \begin{proof}
We assume that $x$ is  balanced  on letters. Consider the line  directed by  the  vector $(\mu_1,\ldots, \mu_d)$ in  $\mathbb R^d$, where the $\mu_i$ are given by \cref{prop:disc}. Again by \cref{prop:disc},  the abelianization vectors   $(|x_{[-n,0]}|_a)_{a \in \mathcal A}$ and  $(|x_{[0,n]} |_a)$, for $n \in {\mathbb N}$, are at $L^1$ distance at most $Cd$  from the  line.
The converse comes from  \cref{prop:disc}.
 \end{proof}

\subsection{Balance and spectral properties}\label{subsec:eigen}
We now relate  balance   with the existence  of  eigenvalues.
 Let $(X,S)$ be a  subshift.
A non-zero complex-valued continuous function $f$ in ${\mathcal C}(X,\mathbb{C})$
is an {\em eigenfunction} for $(X,S)$ if there exists
$\lambda \in  {\mathbb C} $ such that
for all $x$ in $X$,  one has $f(Sx)=\lambda f(x).$ 

The  eigenvalues corresponding to those eigenfunctions
are called the {\em continuous
eigenvalues} of $(X,S)$. If $\theta $ is such that $e^{2i \pi \theta}$ is an   eigenvalue, 
$\theta$ is said to be an {\em additive topological   eigenvalue}.  We  denote by  $E(X,S)$  the  additive  group of additive continuous eigenvalues of $(X,S)$.

We assume that $(X,S,\mu)$ is minimal and uniquely ergodic  (with $\mu$ being the unique shift-invariant  probability measure of $(X,S)$).
  Let  $I(X,S)$ be  the  additive subgroup  of ${\mathbb R}$ generated by   the measures of  cylinders. One has   by  \cite[Proposition 2.3]{BCDLPP:19}:
$$I(X,S) = \left\{ \int f d\mu \,: \, f \in \C(X,{\mathbb Z})\right\}=  \langle\mu[w]  \,: \,   w\in\mathcal{L}(X)\rangle,$$ 
with the notation $\langle\mu[w]  \,: \,   w\in\mathcal{L}(X)\rangle$ standing for the additive  group generated by the $\mu[w]$.
Moreover, one has  $E(X,S) \subset I(X,S)$ by   \cite[Proposition 11]{CortDP:16} and   \cite[Corollary 3.7]{GHH:18}.

 Balance  on letters  provides  a sufficient condition
for  the  additive group generated by  measures of one-letter cylinders to be  included in $E$      under the assumption of unique ergodicity.
We recall  indeed the   following classical   theorem from topological dynamics. Let
 $ {\mathcal C} (X,\mathbb{R})$  stand for the set of continuous functions defined on $X$ with values in ${\mathbb R}$.

\begin{theorem} [Gottshalk-Hedlund's Theorem \cite{Gott.Hed.55}] \label{thm-got}
Let $(X,S)$ be  a minimal  subshift and $f\in {\mathcal C} (X,\mathbb{R})$. 
Then, there exists $x_0\in X$ such that $(S^nf(x_0))_n$ is a bounded sequence if and only if $f$ is a coboundary, i.e.,  there exists a function $g\in   {\mathcal C}(X,\mathbb{R})$ such that $
 g\circ S-g=f.$
\end{theorem}
\begin{remark} \label{rem:int}
Let $(X,S,\mu)$ be  a minimal  and uniquely ergodic  subshift and $f\in {\mathcal C} (X,\mathbb{R})$.  If $f$ is coboundary, then 
$\int  f d\mu=0$,  as a consequence of the ergodic theorem.
\end{remark}

We thus deduce  a classical  folklore result  that allows  the reformulation of balance in spectral terms. 
\begin{proposition}\label{prop-vp}
Let $(X,S,\mu)$ be  a minimal  and uniquely ergodic subshift. Let $w$ be a  factor  of ${\mathcal L}(X)$.
If   $(X,S)$ is  balanced on $w$, then  $\mu[w]$  is an additive  topological   eigenvalue of $(X, S)$.
\end{proposition}
\begin{proof}
We assume that  $(X,S)$ is  balanced on $w$. 
Let $f= {\bf 1}_{w} -   \mu [w] {\bf 1}$, where ${\bf 1}$  is the constant function defined on $X$ taking everywhere the value $1$.
 Since $(X,S)$ is  balanced on $w$,   we can apply  Theorem \ref{thm-got}: there exists $g$ in $ {\mathcal C} (X,\mathbb{R})$ such that 
$f = g - g \circ S$.
Note that $e^{2i\pi {\bf 1}_{[w]}(x)}=1$ for  any $x\in X$.
This yields
$$\exp({2i\pi g\circ S}) = \exp ({2i\pi \mu [w] })  \exp  ({ 2i \pi  g}).$$
Hence  $\exp  ^{ 2i \pi  g}$ is a continuous   eigenfunction    associated with the additive  eigenvalue $\mu [w]$.
\end{proof}

\subsection{Bounded remainder sets}\label{subsec:BRS}

So far, we have mostly  considered  symbolic  dynamical systems made of   sets of infinite words on which the shift map acts.    We now consider  their geometric counterpart in this setting,     namely toral translations. We will  see indeed in  Section \ref{sec:defbilliard} that  billiard words on $d+1$ letters are  codings of toral translations on  the $d$-dimensional torus $\mathbb{T}^{d}$ (see Theorem \ref{thm-rappel}).  We thus  provide  a further interpretation of  balance in arithmetic terms.
For more on toral translations, see e.g.  \cite{Walters}.

Let  $\balpha= (\alpha_1,\dots,\alpha_{d}) $  be  a vector in $(0,1)^d$. We  assume that $\balpha=(\alpha_1, \cdots, \alpha_d)$  is {\em  irrational}, i.e., $ \alpha_1,\dots,\alpha_{d},1$ are rationally independent. 
We consider the   toral translation $T_{\balpha}$   defined on $\mathbb{T}^{d} = \mathbb{R}^{d}/\mathbb{Z}^{d}$ by \[
R_{\balpha}:\  \mathbb{T}^{d} \rightarrow \mathbb{T}^{d}, \quad  {\mathbf x} \mapsto  {\mathbf x}  + (\alpha_1,\dots,\alpha_{d})  \pmod{\mathbb{Z}^{d}}.
\]

This  translation is minimal and uniquely ergodic by the assumption of   irrationality.  Its  unique invariant measure
is the Haar measure $\mu$ on $\mathbb{T}^{d}$. 
We first  recall a classical spectral result (see e.g. \cite{Walters}), with eigenvalues of   the   toral translation $(\mathbb{T}^{d}
,T_{\balpha})$   being  defined similarly as for subshifts (see \cref{subsec:eigen}).
\begin{lemma}\label{rot-eigenvalue}
Let  $\balpha= (\alpha_1,\dots,\alpha_{d}) $  be  a vector in $(0,1)^d$. We  assume that $\balpha=(\alpha_1, \cdots, \alpha_d)$  is {\em  irrational}. Consider the translation  $(\mathbb T^d, T_{\balpha})$. Then, its  additive group   of continuous eigenvalues is of rank $d$. It is equal to
$\langle \alpha_1, \ldots, \alpha_d\rangle:=  \alpha_1\Z+ \cdots+ \alpha_d\Z.$
 \end{lemma}
We now    state an important definition  closely related to the notion of balance.
\begin{definition}
Let $\balpha= (\alpha_1,\dots,\alpha_{d}) $  be  an irrational  vector in $(0,1)^d$.
A bounded remainder set   for $(\mathbb T^d, T_{\balpha})$ is a  measurable set 
$A$  of  $\mathbb T^d$ for which there  exists  $C\geq 0$ such that,  for all $N\in\mathbb N$, we have
$$ | {\Card} \{ 0 \leq n \leq N  \,: \,   R_{\balpha} ^n (0) \in  A \} - N \mu(A) | \leq C.$$
\end{definition}

A bounded remainder  set  $A$ is a    set    having a   bounded discrepancy, in the sense that the difference between the number of visits  under the map $R_{\balpha}$ to this particular  set  $A$ and its expected value is bounded. The  study of bounded remainder  sets started with the work of W.~M.~Schmidt \tcb{in his series of papers} on  irregularities of distributions in the  arithmetic setting of equidistribution and  discrepancy theory. Dynamically, these are sets with bounded  ergodic deviations, i.e.,  the Birkhoff sums  $S_n f_A$  associated with 
 the function $f_A={\mathbf 1}_A -\mu[A] {\mathbf 1}$ (with ${\mathbf 1}_A$ being the characteristic function of the set $A$ and ${\mathbf 1}$
 being the constant function equal to $1$)  are uniformly   bounded with $n$. 
The notion of balance  can be reinterpreted   in symbolic terms  (i.e.,  for  words)   as the fact that the cylinders provided  by  factors are bounded remainder sets for the dynamics  of the shift, by Proposition \ref{prop:disc}.

S. Grepstad and N. Lev have given   in   \cite{GL:15}  a     particularly nice  family  of    bounded remainder sets for  irrational rotations,   as  the  parallelotopes in
   ${\mathbb  R}^d$  spanned by vectors belonging to 
  $ {\mathbb Z}\balpha+ {\mathbb Z}^d$  (for an irrational  rotation by $\balpha=(\alpha_1, \cdots, \alpha_d)$  acting on ${\mathbb T}^d$). 
  This  family generalizes  Kesten's  characterization of   the intervals  that are  bounded  remainder sets  of the  unit circle  for an  irrational rotation 
  by $\alpha$  as     the intervals of length in $\alpha{\mathbb Z}+ {\mathbb Z}$.
\begin{theorem}\label{GL:thm} \cite{HKK:17} \cite{GL:15}
Let $\balpha= (\alpha_1,\dots,\alpha_{d}) $  be  an irrational  vector in $(0,1)^d$.
Any parallelepiped in ${\mathbb R}^d$  spanned by vectors
$v_1, \cdots, v_d$  belonging to  ${\mathbb Z} \balpha+  {\mathbb Z}^d$
is a bounded remainder set for $(\mathbb T^d, T_{\balpha})$. \end{theorem}

The paper  \cite{HKK:17} provides a very elegant proof of this result which is  particularly relevant in the present  context of  one-dimensional cut and project tilings (see \cref{sec:CP}).
  The proof of 
    \cite{GL:15}  is  based on the study of the Fourier transform of the transfer function (for the coboundary equation), whereas  
   the geometric approach    of \cite{HKK:17} relies on the notion of bounded distance  to a lattice: the set  of return times of a toral rotation to a bounded remainder set is proved to  have  a group structure coming from a natural higher-dimensional realization of the problem.

Furthermore, according to  \cite[Corollary 1]{GL:15}, any convex, centrally symmetric polygon in $\R^2$  with vertices belonging to   ${\mathbb Z} \balpha+  {\mathbb Z}^2$ is a bounded remainder set.  The following statement  even provides a characterization for convex polygons being bounded remainder sets. 
\begin{corollary}\label{cor:GL} \cite[Theorem 3]{GL:15}
A convex  polygon in $\R^2$ is a bounded remainder set if and only if it has a center of symmetry and for every pair of parallel edges $e, e'$, the following  two condition hold:
\begin{enumerate}
\item there exist  two points on each edge $e$ and $e'$ related by an element of $\mathbb Z\balpha+\mathbb Z^2$;
\item if the midpoints of the edges  $e$ and $e'$ are not related by an element of $\mathbb Z\balpha+\mathbb Z^2$, then the vectors $e, e'$  belong to $\mathbb Z\balpha+\mathbb Z^2$.\end{enumerate}
\end{corollary}
This statement will play a key role for the proof of \cref{prop:long2}.

\section{Hypercubic billiards}\label{sec:defbilliard}
\subsection{Definitions}\label{subsec:hypercube}
In this section, we discuss   billiards in the hypercube in ${\mathbb R}^{d+1}$ and  their relations with   translations on the torus $\T^d$. For more on combinatorial properties of  hypercubic billiard words, see \cite{Ar.Ma.Sh.Ta.94,Ta,Bary.95,HB:07,Bed:09}.

\medskip

{\bf Billiard trajectories  in the hypercube}
Billiard words are defined  as codings of  trajectories in hypercubic billiards when considering  elastic reflexion on the boundary: the  trajectories are  obtained by moving with constant speed along straight lines until a face of the hypercube is hit.
More formally, consider the unit hypercube $C=[0,1]^{d+1}$ in $\mathbb R^{d+1}$. A billiard ball, i.e., a point mass, moves inside $C$ with unit speed along a straight line until it reaches the boundary $\partial C$ of $C$, then  instantaneously changes direction according to the mirror law, and continues along the new line. Thus, we have a map  $T$ defined on a subset of $\partial C\times \mathbb R^{d+1}$as  follows: $T(m,\btheta)=(m',\btheta')$, with the line $(mm')$ having direction $\btheta$, and  $\btheta'$ being  the image of $\btheta$ by the reflection along the face containing $m'$. Note that the map $T$  is not defined if the line $m+\mathbb R \btheta$ is included in a face of $C$, or if this line intersects $C$ on a face of dimension strictly less that $d$. This is why we consider here a  subset of $\partial C\times \mathbb R^{d+1}$, but it is a subset of full measure.

\medskip

{\bf Billiard subshifts} Let us then label the faces of  the hypercube $C$ by $d + 1$  symbols from the finite alphabet ${\mathcal A}=\{1, \cdots, d+1\}$  such that any  two opposite faces of the hypercube are coded by the same symbol. With the orbit of a point $(m,\btheta)$, we associate the  word  with values in the alphabet ${\mathcal A}$ which is given by the sequence of faces of the billiard trajectory.  We  work with  biinfinite words  by considering  orbits in the past and in the future, since the  map  $T$ is invertible. We denote by $\phi$ the map which sends $(m,\btheta)$ to the  word which codes the billiard orbit.
Then, we have two subshifts.  The first  subshift  $(X_{(m,\btheta)},S)$ is defined as  the closure of the shift orbit of $\phi(m,\btheta)$. The second one,  denoted as $(X_{\btheta},S)$,
  is  defined as the closure of all the orbits of points in the initial direction $\btheta$. 
We will say that the direction $\btheta$ is {\em minimal} if for almost all 
$m\in\partial C$, the orbit of $(m,\btheta)$  under the action  of $T$ is dense in the boundary $\partial C$  of $C$. In this case the two subshifts   $(X_{(m,\btheta)},S)$ and  $(X_{\btheta},S)$ do coincide. From now on, we assume that  the direction $\btheta$ is minimal.

\medskip 
{\bf Unfolding trajectories}
Now a classical tool  is to  unfold  billiard trajectories.   This is the viewpoint underlying  \cref{sec:CP}.
We consider the group $G$ generated by the linear reflections through the faces of $C$ of maximal dimension. This group has
$2^{d+1}$ elements. Then we consider $2^{d+1}$ copies of the hypercube, 
 each of them being the  image of $C$ by one element of $G$. We glue them face by face if $f'=gf$, with  $f'\in gC$ and $ g\in G$. We  then obtain  a fundamental domain for the torus   $\mathbb T^{d+1}=\mathbb R^{d+1}/\mathbb Z^{d+1}$  with the linear   flow acting on it.  See \cref{fig:unfold} for an illustration. 
 
We now consider a section of this linear flow. Let $L_{\btheta}$ stand for the line in $\R^{d+1}$ directed by  the vector $\btheta$.  In other words, travelling along the  line $L_{\btheta}$,  one meets the faces of the unit hypercubes that are located at the grid defined by the set ${\mathbb Z}^{d+1}$ of integer points.  Let  $(e_1, \cdots e_{d+1})$ stand for the canonical basis of ${\mathbb R}^{d+1}$. Then  one codes the sequence of upper faces (of unit hypercubes) that are met by $L_{\btheta}$, where if the line hits a face parallel to the  face with normal vector $e_j$, we code this intersection with the letter $j$. 

 A classical result links the billiard flow in the hypercube $C$ with a translation    in the torus $\mathbb T^{d}$  (such as discussed in \cref{subsec:BRS}).
 For that  purpose, we  recall an explicit way   to obtain  a fundamental domain for the torus $\mathbb T^{d}=\mathbb R^{d}/\mathbb Z^{d}$. 
 
 Define first 
$$
F_i = \{(x_1,\dots,x_d) \in [0,1]^{d+1}\,:\, x_i = 1\} \qquad (1 \le i \le d+1)
$$
to be an \emph{upper face} of the $d+1$-dimensional unit cube,  and similarly 
\[ 
\tilde{F}_i = \{(x_1,\dots,x_d) \in [0,1]^{d+1}\,:\, x_i = 0\}  \qquad (1 \le i \le d+1)
\]
as a \emph{lower face}.

Consider the intersection of $C=[0,1]^{d+1}$ with the hyperplane $L_{\btheta}^\perp$. We obtain a polytope with parallel faces, denoted as $W_{\btheta}$.
 We
let  denote by $\pi_{\btheta}$ the orthogonal projection  along  the line  $L_{\btheta}$ onto  the  hyperplane  $L_{\btheta}^\perp$ with normal vector $\btheta$.
One has  $W_{\btheta}=\pi_{\btheta} (C)$.
We divide  $W_{\btheta}$ into $d+1$ pieces (denoted as  $W_{\btheta} ^{(i)}$, $1  \leq i \leq d+1$), which correspond to the projections of the $d+1$ lower hyperfaces of the unit hypercube, i.e., $W_{\btheta} ^{(i)}=\pi_{\btheta} (\tilde{F}_i)$ for each $i$. For $i=1,\ldots,d+1$,   we define $f_i$  as the  vector  $f_i=\pi_{\btheta}(e_i)$. 
We  then  define on  $W_{\btheta}$  the following  piecewise translation  $E_{\btheta}$ as an exchange of the  $d+1$ pieces $W_{\btheta} ^{(i)}$, as follows: 
\begin{equation} \label{eq:exchange}E_{\btheta}\colon W_{\btheta} \rightarrow W_{\btheta}, \ x \mapsto x + f_i = x + \pi_{\btheta}(e_i) \mbox{ if  } x\in  W_{\btheta} ^{(i)}= \pi_{\balpha}  (\tilde{F}_i).\end{equation}
 In fact, this is the first return map of the linear flow in direction $\btheta$.  See \cref{fig-tore} for an illustration in the case where $d=2$.
 Note that $E_{\btheta} (
\pi_{\balpha}  (\tilde{F}_i))= \pi_{\balpha}  ({F}_i)$, for  each $i$.

Now, let $\Lambda_{\btheta}$  be  the $d$-dimensional lattice
in  the  hyperplane  $L_{\btheta}^\perp$ generated by  the $d$ vectors $f_i- f_{d+1}$, for 
$i=1,\ldots,d$. We consider the  translation $T_{\btheta}$ defined on  $L_{\btheta}^\perp/\Lambda_{\btheta} $ as $ {\mathbf x} \mapsto  {\mathbf x}+f_1$.
It turns out that the   exchange of pieces $E_{\btheta}$  on 
the domain   $W_{\btheta}=\pi_{\btheta}([0,1])^{d+1}$ is   measurably isomorphic to  the   translation $T_{\btheta}$.
This is  a direct consequence  of the fact that  the set of  translates of $W_{\btheta}$ of the form   $W_{\btheta} + \sum_{i}^{d+1} \Z f_i$  tiles the hyperplane   $L_{\btheta}^\perp$
periodically.
In other words, the coding of the billiard orbit in direction $\btheta$ is equal to the 
  coding of the map $E_{\btheta}$ induced by  the partition of the domain $W_{\btheta}$ by  the  $d+1$  subsets $W_{\btheta}^{(i)}$.
  For more details, see e.g.  \cite{BCCSY}.

 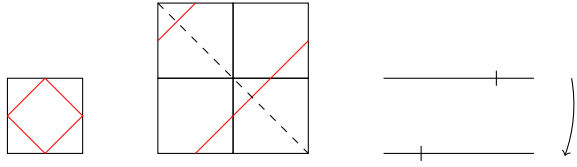
\begin{figure}
 \begin{center}
\begin{tikzpicture}
\draw (0,0)--(1,0)--(1,1)--(0,1)--cycle;
\draw[red] (.5,0)--(1,.5)--(0.5,1)--(0,0.5)--cycle;

\draw (2,0)--(3,0)--(3,1)--(2,1)--cycle;
\draw (3,0)--(4,0)--(4,1)--(3,1)--cycle;

\draw (2,1)--(3,1)--(3,2)--(2,2)--cycle;
\draw (3,1)--(4,1)--(4,2)--(3,2)--cycle;
\draw[red] (2.5,0)--(4,1.5);
\draw[red] (2,1.5)--(2.5,2);
\draw [dashed] (2,2)--(4,0);

\draw (5,0)--(7,0);
\draw (6.5,0.9)--++(0,.2);
\draw (5,1)--(7,1);
\draw (5.5,-0.1)--++(0,.2);
\draw   [->] (7.5,1) arc(10:-20:2);
\end{tikzpicture}
\caption{One billiard trajectory in the square (left). One translation on the torus defined by unfolding with respect to four squares (middle). The first return on the diagonal is an exchange of two intervals (right).} \label{fig:unfold}
\end{center}
\end{figure}


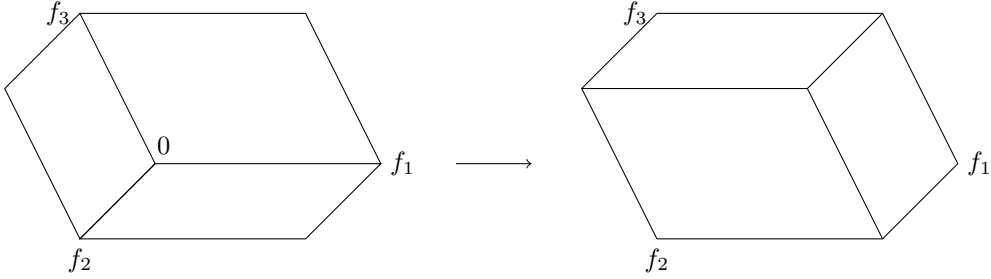
\begin{figure}
\begin{tikzpicture}
\draw (4,1) node[right]{$f_1$};
\draw (1,1) node[right,above]{$\ \ 0$};
\draw (0,3) node[left]{$f_3$};
\draw (0,0) node[right,below]{$f_2$};
\draw (0,0)--(3,0)--++(1,1)--++(-3,0)--cycle;
\draw (0,0)--(-1,2)--++(1,1)--++(1,-2)--cycle;
\draw (0,3)--++(3,0)--(4,1);

\draw[->] (5,1)--(6,1);
\end{tikzpicture}
\hspace{.5cm}
\begin{tikzpicture}
\draw (4,1) node[right]{$f_1$};
\draw (0,3) node[left]{$f_3$};
\draw (0,0) node[right,below]{$f_2$};
\draw (0,0)--(-1,2)--(0,3);
\draw (0,3)--++(3,0)--(4,1);
\draw (0,0)--(3,0)--(4,1);
\draw (3,3)--(2,2)--(-1,2);
\draw (2,2)--(3,0);
\end{tikzpicture}
\caption{The translation  on the torus $\mathbb T^2$ represented by the fundamental domain $W_{\btheta}=\pi_{\btheta}[0,1]^2$ and the domain exchange $E_{\btheta}$ acting  on it.}\label{fig-tore}
\end{figure}

We recall classical  facts about the subshift $(X_{\btheta},S)$.
\begin{theorem}\label{thm-rappel}
We consider the  direction ${\btheta}= (1,\theta_1, \ldots, \theta_d) \in \R^{d+1}_+$.
\begin{enumerate}
\item If ${\btheta}$ is an irrational direction, then, for every $x\in X_{\btheta}$, the set of  the  points $m\in \partial{C}$ such that their orbit is coded by $x$ contains  exactly one point.
\item The subshift $(X_{\btheta},S)$ is minimal if and only if $\btheta$ is an irrational direction.
\item The subshift $(X_{\btheta},S)$ is uniquely ergodic if and only if $\btheta$ is an irrational direction. 
\item The subshift  $(X_{\btheta},S)$ is measurably isomorphic  to the translation  $(T_{\balpha},\T^{d})$, with $$\balpha=\left(\frac{\theta_1}{1+\theta_1+\cdots+\theta_d},\cdots, \frac{\theta_d}{1+\theta_1+\cdots+ \theta_d}\right).$$
\item Letters are balanced in $X_{\btheta}$. The frequencies of letters    are given by   $ \mu[1]=\frac{ 1}{ 1 + \sum_j \theta_j}$, and  $ \mu[i]=\frac{ \theta_i}{ 1 + \sum_j \theta_j}$ for all $i=1, \ldots, d$.
\item   Assume that  ${\btheta}$ is an irrational direction. When $d=2$, there are  7 factors of length $2$ in ${\mathcal L}(X_{\btheta})$.
\end{enumerate}
\end{theorem}

Let us  briefly  give a  few bibliographic   elements  concerning  these results. The coding by the same letter for parallel faces is made such that the coding map  is almost everywhere one to one; see \cite{Ga.Kr.Tr}. Thus the associated subshift
$(X_{\btheta},S)$  is also minimal and uniquely ergodic under the  irrationality hypothesis  made on the direction $\btheta$. 
Letters are balanced by construction: abelianizations of infinite words in the subshift are within bounded distance of the trajectory in ${\mathbb R}^d$ of  the translation (see \cref{lem:dist}). The complexity function has been studied in the following papers \cite{Ar.Ma.Sh.Ta.94, Bary.95, Bed.03, Bed:09}.

\begin{remark}
By Theorem \ref{thm-rappel}, the subshift $(X_{\btheta},S)$ is minimal if and  only if it is uniquely ergodic. In this case, for almost every point $m\in \delta C$, the coding word $x$ of its orbit define the same subshift, i.e., $X_x=X_{\btheta}$.
\end{remark}

\subsection{More on the cubic  case} \label{subsec:cubic}
We now focus on the cubic case $d+1=3$.
We  use the notation of  \cref{subsec:hypercube}. Let   $\btheta=\begin{pmatrix}1\\ \theta_1\\ \theta_2\end{pmatrix}$ be an irrational vector in  $[0,1]^3$.  We recall that  the cube $[0,1]^3$ projects, by the projection $\pi_{\btheta}$,     onto the hexagon  $W_{\btheta} \subset  (\mathbb R\btheta)^\perp$. Moreover,  for every integer $n\geq 1$, there is a partition of $W_{\btheta}$ into cells corresponding to the billiard factors  of length $n$
given by $$W_{\btheta}  ^{(w_1)} \cap   E_{\btheta} ^{-1}  W_{\btheta}  ^{(w_2)} \cap  E_{\btheta} ^{n-1}W_{\btheta}  ^{(w_n)}, \mbox{ for }  w=w_1 \cdots w_n \in {\mathcal
L}(X_{\btheta})  .$$  By unique ergodicity (by \cref{thm-rappel}), the Lebesgue measure of these cells is  then   proportional to the frequency of  factors $\mu[w]$
(up to a renormalizing constant).
We recall that the hexagon  $W_{\btheta}$ is  a fundamental domain for the two-dimensional  torus $L_{\btheta}^\perp/\Lambda_{\btheta}$, with the  the lattice $\Lambda_{\btheta}$ being  generated by $f_1-f_3, f_1-f_2$. Now we consider the translation by $f_1$ on the torus $L_{\btheta}^\perp/\Lambda_{\btheta}$. A simple computation shows that the projection $\pi_{\btheta}$ on $L_{\btheta}^\perp$ has the matrix (expressed in the canonical basis of ${\mathbb R}^3$)
$$\frac{1}{1+\theta_1^2+\theta_2^2}\begin{pmatrix}\theta_1^2+\theta_2^2&-\theta_1&-\theta_2\\ -\theta_1&1+\theta_2^2&-\theta_1\theta_2\\ -\theta_2&-\theta_1\theta_2&1+\theta_1^2
\end{pmatrix}.$$
We deduce the three vectors $f_1, f_2, f_3$. Note  that $f_1+\theta_1f_2+\theta_2f_3=0$.  Hence  the following expression defines the vector of translation in term of the basis of the lattice, namely  $$f_1=\frac{\theta_1}{1+\theta_1+\theta_2}(f_1-f_2)+\frac{\theta_2}{1+\theta_1+\theta_2}(f_1-f_3).$$

On ${\mathbb T}^2$, the translation from Theorem  \ref{thm-rappel}  is  $ {\mathbf x} \mapsto   {\mathbf x} + (\frac{\theta_1}{1+\theta_1+\theta_2}, \frac{\theta_2}{1+\theta_1+\theta_2})$. 
 This translation    acting on the  two-dimensional  torus $\mathbb T^2$ is  measurably isomorphic  to  the  exchange of three rhombi $(W_{\btheta}, E_{\btheta})$, from \cref{eq:exchange}, depicted in  Figure \ref{fig-tore}. 

\subsection{Cut and project sets and tilings }\label{sec:CP}


We now describe cut and project sets of codimension one and link them to billiard words.
For that purpose, we  first recall the notion of  a (codimension one) cut and project set in the simplest setting\footnote{The general definition  of a cut and  project set   is stated   for    locally compact Abelian groups.}. For more  on cut and project sets, see \cite{CRM:2000,BaakeGrimm,MAP:15}. 

\medskip 
{\bf Cut and project sets}  Let $\btheta=(1,\theta_1, \ldots, \theta_d) \in \R^{d+1}_+$.  We consider  a decomposition of $\R^{d+1}$ as a sum of two direct subspaces,
namely   the line  $L_{\btheta}$ directed by the  vector $\btheta \in \R^{d+1}_+$ (called {\em physical space}), and its orthogonal hyperplane $ L_{\btheta}^\perp$ (called {\em internal space}). 
We assume  that  the coordinates of  $\btheta$ are independent over $\mathbb Q$,  i.e., the direction  $\btheta$ is irrational.
We now consider  sets  of points  that  are  formed by projections together with a way of selecting points. 
The selection is done thanks to an acceptance window $W$ that lives in the internal space $L_{\btheta}^\perp$
  (usually assumed to be a polytope). 
  
 We recall that   $\pi_{\btheta}$ stand for  the orthogonal projection  along  the line  $L_{\btheta}$ onto  the  hyperplane  $L_{\btheta}^\perp$.   We let  denote by  $\pi_{\btheta}^\perp$ the orthogonal projection onto  $L_{\btheta}$.
Observe that the irrationality  assumption on   $\btheta$  implies  that 
$\pi _{\btheta}(\mathbb Z^{d+1})$ is dense in  $L_{\btheta}^\perp$, and that
  the restriction of $\pi_{\btheta}^{\perp}$ to $\mathbb Z^{d+1}$ is one-to-one.
  The second point is especially important: if an interval in $L_{\btheta}$ has a length which is in $\pi_{\btheta} (\mathbb Z^{d+1})$, then it ``lifts'' to a unique  vector in $\mathbb Z^{d+1}$.

In the following, we always let  the window $W_{\btheta}= \pi_{\btheta} ([0,1]^{d+1})$  be the projection on $L_{\btheta}^\perp$ of the unit hypercube $[0,1]^{d+1}$. Let $m \in L_{\btheta}^\perp $  be such that  $\pi^\perp(m+\mathbb Z^{d+1})$ does not intersect the boundary of $W_{\btheta}$ in $L_{\btheta}^\perp$. 
Such a point is said  \emph{generic}.
Given a generic $m \in W_{\btheta}$, the \emph{hypercubical $d+1$-to-$1$ cut and project} set  $\Lambda_m$ of parameter $m$ is defined as:
\[
 \Lambda_m = \bigl\{ \pi_{\btheta}^{\perp}(m+n) \in L_{\btheta}  \,: \,  n \in \mathbb Z^d, \ \pi_{\btheta}(m+n) \in W \bigr\}.
\]
For an illustration in the case $d=1$,  see  \cref{fig:line}.
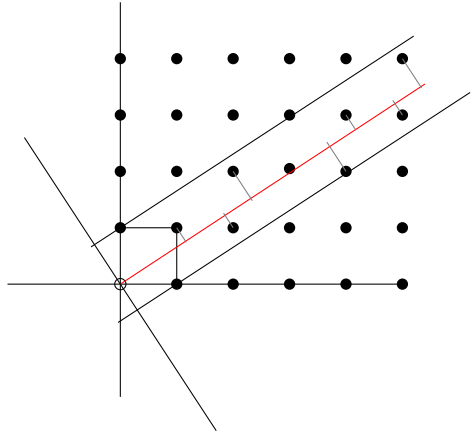
\begin{figure}
\begin{center}
\begin{tikzpicture}[scale=0.75]
\draw (-2,0)--(5,0);
\draw (0,-2)--(0,5);
\fill (1,0) circle(.1); 
\fill (2,0) circle(.1); 
\fill (3,0) circle(.1); 
\fill (4,0) circle(.1); 
\fill (5,0) circle(.1); 
\fill (1,1) circle(.1); 
\fill (1,2) circle(.1); 
\fill (1,3) circle(.1); 
\fill (1,4) circle(.1); 
\fill (2,1) circle(.1); 
\fill (2,2) circle(.1); 
\fill (2,3) circle(.1); 
\fill (2,4) circle(.1); 
\fill (3,1) circle(.1); 
\fill (3,2.05) circle(.1); 
\fill (3,3) circle(.1); 
\fill (3,4) circle(.1); 
\fill (4,1) circle(.1); 
\fill (4,2) circle(.1); 
\fill (4,3) circle(.1); 
\fill (4,4) circle(.1); 
\fill (5,1) circle(.1); 
\fill (5,2) circle(.1); 
\fill (5,3) circle(.1); 
\fill (5,4) circle(.1); 
\fill (0,1) circle(.1); 
\fill (0,2) circle(.1); 
\fill (0,3) circle(.1); 
\fill (0,4) circle(.1); 
\draw (0,0) circle(.1);
\draw (1,0)--(1,1)--(0,1);
\draw[red] (0,0)--(5.4,3.55);
\draw[] (0,0)--(-1.7,2.6);
\draw[] (0,0)--(1.7,-2.6);
\draw (0,1)--++(-.52,-.34);
\draw (1,0)--++(-1.04,-.68);
\draw (0,1)--++(5.2,3.4);
\draw (1,0)--++(5.2,3.4);

\draw[gray] (1,1)--++(.17,-.26);
\draw[gray] (2,1)--++(-.17,.26);
\draw[gray] (2,2)--++(.34,-.52);
\draw[gray] (4,3)--++(.17,-.26);
\draw[gray] (4,2)--++(-.34,.52);
\draw[gray] (5,4)--++(.34,-.52);
\draw[gray] (5,3)--++(-.17,.26);
\end{tikzpicture}
\caption{In red the line $L_{\btheta}$, and the set $\Lambda_m$ for 
$m=0$, which yields a tiling of the   line  $L_{\btheta}$, with $d=2$.} \label{fig:line}
\end{center}
\end{figure}

\medskip

{\bf Subshift  generated by  a cut and project set}
 Let $m$ be a generic point. We follow   \cite{For.Hunt.Kellen.02}.
There is a unique subset $\widetilde{\Lambda}_m$ of $m + \mathbb Z^{d+1}$ which projects onto $\Lambda_m$; the set  $\widetilde{\Lambda}_m$ is included in the strip $W_{\btheta}  \times L_{\btheta} \subset \R^{d+1}$.  Moreover, the set $\Lambda_m$ partitions $L_{\btheta}$ into finitely many intervals having 
exactly $d+1$ different interval lengths, each corresponding to the length of  $\pi_{\btheta}^{\perp}(e_i)$, for  $e_i$ vector  of the canonical basis of $\mathbb R^{d+1}$, $i=1,\ldots, d+1$. 
In other words,  $\widetilde{\Lambda}_m$ defines a ``staircase'' (i.e., a broken line) inside $W_{\btheta} \times L_{\btheta}$, made of concatenated translates of basis vectors.
Therefore, $\Lambda_m$ can be coded into a   word $x(m)  \in \{1,\ldots, d+1\}^{\Z}$, which generates a minimal subshift $X_m$ which does not depend on the parameter $m$ (it depends only on the line $L_{\btheta}$). 
It turns out that this subshift  is in fact   the billiard shift $X_{\btheta}$ defined in \cref{subsec:hypercube}.

\medskip

{\bf  Shift map and $\mathbb Z$-action}
There is an $\mathbb R$-action by translation on $\Lambda_m$ (we identify $L_{\btheta}$ with ${\mathbb R}$).
 This induces a $\mathbb Z$-action via the first return map to $0$. Indeed, 
the   smallest  $t >0   $ such that $\Lambda_m - t$ contains $0$ is a cut
 and project set for a certain parameter $\sigma(m)$,  i.e., 
 $\Lambda_{\sigma(m)}=\Lambda_m-t$.
In fact, if $x(m)$ is the  biinfinite word coding $\Lambda_m$, then $\sigma(m) = m + \pi_{\btheta} (e_{x_0})$. 
This defines a $\mathbb Z$-action on the set   of generic points of the window $W_{\btheta}$. Moreover $m$ is generic if and only if $\sigma(m)$ is generic. Via this identification, the symbolic shift map $x \mapsto Sx$ is  related  to the  map $\sigma$ acting   on the set of generic points of  $W_{\btheta}$ as follows:
 if $x(m)$ is the word coding  $\Lambda_m$, then $Sx(m)$ codes  $\Lambda_{\sigma(m)}$.

The question is  now to find out how much bigger is $X_{\btheta}$ than the set of  words  $x(m)$ obtained  by considering  generic points  $m$ of $W$. It turns out 
that it is not that bigger: 
each non-generic point  $m$ of $W_{\btheta}$ corresponds to  a finite and bounded amount of words in $X_{\btheta}$, as stated below.
\begin{lemma} \cite{For.Hunt.Kellen.02} \label{lem:eta}
 Let $\btheta=(1,\theta_1, \ldots, \theta_d) \in \R^{d+1}_+$ assumed to be irrational.
 The dynamical  systems $( X_{\btheta},S)$ and $(W_{\btheta}, \sigma)$ are measurably conjugate, i.e., 
there exists a  map  $\eta: X_{\btheta} \rightarrow W_{\btheta}$, that is 
 one-to-one precisely on the preimage of the set of generic points  of $W_{\btheta}$: 
$$
\begin{CD}
X_{\btheta} @>S >> X_{\btheta}\\
@VV\eta V @VV\eta V\\
W _{\btheta}@>\sigma>> W_{\btheta}
\end{CD}
$$
Moreover, the cardinality of $\eta^{-1}(\{m\})$ is bounded and only depends on the dimension $d+1$ of the space~$\R^{d+1}$.

\end{lemma}

Note that continuous functions defined  on $W_{\btheta}$ can be  pulled back to continuous functions on $X_{\btheta}$; however there are more continuous functions on $X_{\btheta}$ than the pullbacks from $W_{\btheta}$ (for example, the subshift $X_{\btheta}$ is totally disconnected, it has therefore many locally constant functions and the map $\eta$ is not  a topological conjugacy,   it is not a homeomorphism).
This almost-identification of $X_{\btheta}$ and $W_{\btheta}$  allows us to have a geometric interpretation of statements on the subshift $X_{\btheta}$ (see \cref{sec:cohobi}).


\section{Cohomology  and balance} \label{sec:coho}
The main concept   for proving \cref{theo:main} is {cohomology}, especially some cohomology classes defined  by Clark and Sadun  \cite{CS:03,CS:06}
for tiling spaces and identified in the cut and project setting by Kellendonk and Sadun \cite[Section 6]{Kell.Sad.17}. We focus namely  on the   class of  asymptotically negligible cocycles. Note that corresponding notions  exist  in the symbolic  dynamics setting, which  are  strongly related to dimension groups for subshifts
(a dimension group is an ordered abelian group which is related to the orbit structure of the  subshift).  
For more on the cohomology of tilings, see for instance \cite{Kellendonk,Kell.Putn.06,Sad.07,Sad.08,AnRobin} and for  more  on the symbolic view point, see  e.g. \cite{GPS:95,DuPerrin}. 

\subsection{First cohomological definitions}
Cohomology classes can be represented by functions on $X$ or on $X \times \mathbb Z$, which satisfy specific properties. 
In all that follows,  $\bA$ stands either for  $\mathbb R$ or $\mathbb Z$.

\begin{definition}\label{def:cochain}
Let $(X,S)$ be a subshift.
 A  {\em cochain} with values in $\bA$ is a function $\phi: X \times \mathbb Z \rightarrow \bA$.
 It is a \emph{cocycle} if it satisfies the  following  condition, for all $x\in X$ and all $m,n \in \Z$: $$\phi(x, n+m) = \phi(x,n) + \phi(S^n x, m).$$
 A cochain is a \emph{coboundary} if there exists a function $g: X \rightarrow \bA$ such that,  for all $x\in X$ and all $n \in \Z$, one has  
 \begin{equation} \label{eq:cob}
 \phi(x,n) = g(S^n x) -  g (x).
 \end{equation}
 If we need the precision, we say that $\phi$ is the coboundary of $g$.
\end{definition}

 \begin{example}\label{ex:factorcounting}
 Let $(X,S)$  be a subshift defined over the alphabet $\A$. Let $w \in \A^*$.
 We recall that    the notation ${\mathbf 1}_w$     stands  for the characteristic function  of the cylinder $[w]$  on $X$.
We define the factor-counting  cochain  $\psi_w$   as 
  $$\psi_w (x,n) = \Card \bigl\{ i \in [0, n-1]  \,: \,  x_i \cdots x_{i+|w|-1} = w\bigr\}  \mbox{ for } n   \geq 0$$
   $$\psi_w (x,n) = \Card \bigl\{ i \in [n, -1]  \,: \,  x_i \cdots x_{i+|w|-1} = w\bigr\}  \mbox{ for } n <0$$
   for all  $x$ in $X$  and  for all $n \in {\mathbb Z}$.
   One checks that   $ \psi_w (x,n) =S_n ({\mathbf 1}_w) (x)$, for all  $x$ in $X$  and  for all $n$,
which implies that the   cochain $\psi_w$ is in fact  a cocycle.
 \end{example}

\begin{remark} Note that all  coboundaries are automatically cocycles. 
Observe also that the one-dimensional  tilings of   the line $E_{\btheta}$ induced by the billiard words   in $X_{\btheta}$ (as  evoked in \cref{sec:CP}) provide  a decomposition of 
 the line   into $0$-cells (vertices) and  $1$-cells (edges/tiles). A $k$-cochain is an assignment of a real number to each $k$-cell.
Technically, we should call cochains ``$1$-cochains'', while the function $g$  shoud be called   ``$0$-cochain'', in order to be able to say
 that  ``a $1$-coboundary is the coboundary of a $0$-cochain''. Since we only work in degree $0$ and $1$  (we work with words and one-dimensional tilings), we have chosen   not  to    keep 
 the reference to  the degree  in the  terminology  used in the present paper.  \end{remark} 
 In the symbolic dynamics  literature, it is common to see coboundaries defined in terms of  functions   on $X$ with  values in $\bA$. In fact,  there is  a one-to-one correspondance  between   functions and cocycles (as considered  in \cref{def:cochain}).
Indeed,  given  a  function $f: X \rightarrow \bA$,   we  associate with it  the cochain  $\phi$ defined  as follows
    for all $x \in X$:   $\phi(x,1)= f(x)$ and for all $n \in \Z$
 \begin{equation}\label{eq-cocycle-base}
 \begin{array}{ll}
  \phi(x,n) & =f(x) + f\circ S(x) + \cdots + f \circ  S^{n-1}(x), \mbox{ for } n \geq 0,\\
 \phi(x,n) &=f \circ S^{-1}x+  \cdots + f \circ  S^{n}(x), \mbox{ for } n <0.
 \end{array}
 \end{equation}
In other words,  for $n \geq 0$,  $ \phi(x,n)$   is equal to   the value at $x$ of  the Birkhoff sum  $S_n f$of order $n$ associated with $f$.   
  One then  checks that  the cochain $\phi$ satisfies the  cocyle relation. This thus  allows us to extend the notion of  a coboundary to functions.
   \begin{definition} Let $(X,S)$ be a subshift.
      Let ${\mathcal F} (X,\bA)$ stand for the set of  functions  defined on $X$ with  values in $\bA$.
A function   $f  \in {\mathcal F} (X,\bA)$ is said to be a coboundary if there exists  $g$ in ${\mathcal F} (X,\bA)$
such that 
\begin{equation}\label{eq:cob2}
f=g \circ S-g.
\end{equation}
   The set of coboundaries  of  ${\mathcal F}(X,A)$ is denoted as  $\partial  {\mathcal F}(X,A)$.
   \end{definition}

 Two functions $f,g$
in  ${\mathcal F}(X,A)$ are said to be {\it cohomologous} if $f-g$ is a coboundary.  
This  provides  an equivalence relation  on  ${\mathcal F}(X,A)$.   This  equivalence relation can also
 be defined on the  set of   cochains that are cocycles:
two  cocyles  are  \emph{cohomologous}  if their difference is a coboundary. 

Note that when  restricting to continuous functions and to  $\bA={\mathbb Z}$, we recover the  group part   $ {\mathcal C}(X,{\mathbb Z})/\delta  {\mathcal  C}(X,{\mathbb Z})$ of the  dynamical   dimension group;  see for instance
\cite{GPS:95}. 

\subsection{Strong  and  asymptotically negligible cocycles} \label{subsec:an}
We now define two notions of cohomologies  depending on topological properties of the associated  functions,  both in the setting of cocycles   and functions.  Examples  will be provided  below and in  \cref{subsec:an2}. A crucial point is that   we consider $\bA=\R$.

\begin{definition}\label{def:coho}
Let $(X,S)$ be a subshift.
\begin{itemize} 
\item  A cocycle  is  said to be    a  {\em strong cocycle} if, for all $n \in {\mathbb Z}$,  the following map is  {\em locally constant}:
 $$\begin{array}{ccc}
 X&\to&\R\\
 x& \mapsto& \phi(x,n). 
 \end{array}$$

We   define    $H^1_{\tt{strong}}(X,\mathbb R)$ as the     group  of classes of  locally constant  cocycles  modulo      coboundaries of locally constant  functions (such as defined in \cref{eq:cob}),
and  analogously $\cH^1_{\tt{strong}}(X,\mathbb R)$ as the     group  of classes of  locally constant functions, 
 modulo      coboundaries of locally constant   functions (such as defined in \cref{eq:cob2}).

 \item  
A  strong cocycle  is said to be  \emph{asymptotically negligible} if it is  a  coboundary of a continuous function.
Analogously, a locally constant  function is said to be  {\em asymptotically  negligible} if  it is  a  coboundary of a continuous function.

 We will use the  notation $H^1_{\tt{an}}(X,\mathbb R)$   for  the  group  generated  by the classes of  asymptotically negligible cocyles,  modulo      coboundaries of locally constant  functions,
 and  $\cH^1 _{\tt{an}}(X,\mathbb R)$ for  the group  generated by     asymptotically negligible  functions,  modulo      coboundaries of locally constant   functions. 

\end{itemize}
\end{definition} 
Note that with the previous notation, $\cH^1 _{\tt{an}}(X,\mathbb R)$ is equal to the classes of elements of $ \delta {\mathcal C } (X, {\mathbb R})$, where $ \delta {\mathcal C } (X, {\mathbb R})$
stands for the set of continuous functions defined on $X$ with values in ${\mathbb R}$.

Considering locally constant functions in the  notion of  strong cohomology  is related to  strong pattern equivariance in the tiling literature \cite{Sadun:16}, whereas   considering  continuous functions corresponds to   weakly pattern equivalent functions.
\begin{remark}
We recall that a  continuous function with values in ${\mathbb Z}$ is locally constant. Moreover, if a  continuous function with values in ${\mathbb Z}$  is a coboundary,  then it is  the coboundary of some continuous function \cite[Proposition 3.1.3]{DuPerrin}. 

Consider now the case of functions taking real values. An   asymptotically  negligible function $f =g\circ S- g$, with $g$ continuous with values in ${\mathbb R}$,     has no reason to be  be locally constant   (this would be the case if
 functions  would   take  integral  values).   It thus makes sense  to consider  classes   in $\cH_{\tt{strong}}(X,{\mathbb R}$) of  asymptotically negligible elements (they are not    necessarily   reduced to the 
 zero class  when considered modulo coboundaries of   locally constant functions). 
 \end{remark}

\begin{lemma}  \label{lem:factorcounting} Let $(X,S)$ be  a  subshift on ${\mathcal A}$.  
For every $w \in  {\mathcal A}^*$ of length $k$, 
the  factor-counting map $\psi_w$ defined by
\begin{equation}\label{eq-cocycle-base2}
 \psi_w (x,n) = \Card \bigl\{ i \in [0, n-1]  \,: \,  x_i \cdots x_{i+k-1} = w\bigr\}
\end{equation}
 is a strong cocycle.
\end{lemma}
\begin{proof} We have seen  that  the cochain $ \psi_w (x,n) =S_n ({\mathbf 1}_w) (x)$, for all  $x$ in $X$  and  for all $n$, is a cocycle (by \cref{ex:factorcounting}).
Now for all $n$,  the map $ x \mapsto \psi_w (x,n)$ is a continuous function  that  takes integer values.  It is thus locally constant and  we have a strong cocycle.
\end{proof}

Let us now  revisit the results of \cref{subsec:eigen}.
By direct application of Gottshalk-Hedlund's Theorem  (see \cref{thm-got}) we deduce the following  characterization of  asymptotically negligible cocyles.
\begin{corollary}\label{cor-angh} Let $(X,S)$ be  a minimal  subshift.
 A  strong cocycle  $\phi: X \times \mathbb Z \rightarrow \R$ is asymptotically negligible if and only if for all (equivalently for one) $x$ in $X$, the map $n \mapsto \phi(x,n)$ is bounded.
\end{corollary}

Let us investigate  what  balance means  in terms of cohomology. 

Let  ${\mathbf l}: X \times \mathbb Z \rightarrow \R$ stand for the length-counting cocycle defined ${\mathbf l} (x,n) = n$ for all  $x \in X$ and $n \in   \mathbb Z $. Note that it is not an asymptotically negligible cocycle.

\begin{proposition}\label{cor-balanced-asympt}
Let $(X,S,\mu)$ be  a minimal  and uniquely ergodic subshift. Let $w$ be a finite factor of its language ${\mathcal L}(X)$.
The subshift   $(X,S)$ is  balanced on $w$ if and only if  the cocycle $\psi_w - \mu[w] \mathbf l$  is asymptotically negligible, which is  in turn equivalent to $S_n({\mathbf 1}_w   -  \mu[w]{\mathbf 1} )$ being bounded uniformly in $n$.

\end{proposition}

\begin{proof}
We know from  Lemma \ref{lem:factorcounting}  that $\psi_w $ is a strong cocyle.  This is also  the case of  $\psi_w  - \mu[w] \mathbf l$.
By Corollary \ref{cor-angh} the symbolic discrepancy $S_n({\mathbf 1}_w   -  \mu[w] {\mathbf 1})$ is bounded uniformly in $n$ if and only if the cocycle is asymptotically negligible.
We then  use \cref{prop:disc}.
 \end{proof}

\section{On the  cohomology group of the billiard}\label{sec:cohobi}

We first  recall a    key fact about the cohomology group of the hypercubic billiard  which is central for the proof of \cref{theo:main}.
\begin{theorem}\cite[Theorem~6.10]{For.Hunt.Kellen.02}\label{prop-jul}
Let $\btheta \in  {\mathbb R}^{d+1}_+$ be an irrational  direction. Let  $(X_{\btheta},S)$    be the billiard  shift associated with the direction $\btheta$.
Let $L_{\btheta}^\perp$ stand for the hyperplane with normal vector $\btheta$.
If  $d\geq 2$,  then the  cohomology group  $H^1_{\tt{strong}}(X_{\btheta},\mathbb R)$   of  the subshift $(X_{\btheta},S)$  (as defined in \cref{def:coho}) is not finitely generated.
\end{theorem}

The aim of  this section is to discuss in more details  asymptotically negligible elements. This is the second ingredient required for the proof of \cref{theo:main}.

\subsection{Examples of asymptotically negligible elements for billiards} \label{subsec:an2}

We first  provide in this section  a geometric way for  constructing  asymptotically negligible elements for an irrational   dimensional hypercubic billiard in ${\mathbb R}^{d+1}$.

We use the following notation:  if $a \in \{1,2,\ldots, d+1\}$, $e_a$ stands for  the corresponding basis vector in $\mathbb R^{d+1}$, and  if $w \in  \{1,2,\ldots, d+1\}^*$, $e_w = \sum_i e_{w_i}$. 
Let $\btheta \in  {\mathbb R}^{d+1}_+$ be an irrational  direction. 

Let $ (X_{\btheta},S)$    be the billiard  shift associated with the direction $\btheta$.
Given a vector $v \in E_{\btheta}^\perp$,    the cocycle $\phi_v$ is defined   on $X_{\btheta} \times \mathbb Z$ with  values in ${\mathbb R}$  as follows:  for all $x \in X_{\btheta}$, one has 
\[
 \phi_v(x, n) = \langle v, e_{x_0 x_1 \cdots x_{n-1}} \rangle, \mbox{ for all }  \ n  \geq 0,
\]
\[
 \phi_v(x, n) = \langle v, e_{x_{-n}  \cdots x_{-1}} \rangle, \mbox{ for all } \ n  <0.\]

This cocycle is defined  with respect to the  geometric realization of  prefixes  of  elements of  $X$ obtained by ``abelianizing'' words $w$ in ${\mathcal A}^*$  by vectors $(|w_a|)_{a \in {\mathcal A}}$  in  ${\mathbb R}^{d+1}$.
The corresponding function is $$f_v:X_{\btheta} \rightarrow {\mathbb R}, \ x \mapsto  \langle v, e_{x_0}\rangle.$$

\begin{lemma} \label{lem:phiv}
Let $\btheta \in  {\mathbb R}^{d+1}_+$ be an irrational  direction. Let  $(X_{\btheta},S)$    be the billiard  shift associated with the direction $\btheta$.
Let $L_{\btheta}^\perp$ stand for the hyperplane with normal vector $\btheta$.
For any  $v \in L_{\btheta}^\perp$,  the map $\phi_v$ is an  asymptotically negligible strong cocycle.
\end{lemma}

\begin{proof}
We have to prove that   the map $\phi_v$  is a locally constant cocyle that is  the coboundary of a continuous function. 
First, we prove that   the map $\phi_v$  is a cocycle.   Let $x \in X_{\btheta}$ and $n,m \in {\mathbb N}$. One has
$$\phi_v(x,n+m)= \langle v, e_{x_0 x_1 \cdots x_{n+m-1}} \rangle= \langle v, e_{x_0 x_1 \cdots x_{n-1}} \rangle+ \langle v, e_{x_n x_{n+1} \cdots x_{n+m-1}} \rangle,$$
i.e.,
$$\phi_v(x,n+m)=\phi_v(x,n)+\phi_v(S^nx,m).$$
One similarly checks the cocyle relation for $m,n \in \Z$.
In addition, for any  fixed $n$,  $ \phi_v (x,n)$  only depends on $x_{[-n,-1]}$ ($n<0$) or on  $x_{[0,n-1]}$ ($n\geq0$): it is therefore locally constant, it is thus a strong cocycle.

Let us  check that the map  $\phi_v$  is the coboundary of the function $b$ defined  on $X_{\btheta}$ as:
\[
 b(x) = \langle v, \eta(x) \rangle, \mbox{ for all }  x  \in X_{\btheta},
\]
where $\eta$ is the almost everywhere one-to-one factor map $X_{\btheta} \rightarrow W_{\btheta}$ defined in \cref{sec:CP} (see \cref{lem:eta}).
Let $x \in X_{\btheta}$. We  first assume $ n\geq 0$.
One has 
$$
 b(S^nx)-b(x)  = \langle v, \eta(S^nx) \rangle- \langle v, \eta(w) \rangle=   \langle v, \sigma^n(m) \rangle- \langle v, m \rangle,
$$
with the map $\sigma$ being defined in \cref{sec:CP}.
Since $\sigma(m)=m+\pi_{\btheta}(e_{x_0})$, we deduce that  $\sigma^n(m)=m+\pi_{\btheta}(e_{x_0\dots x_{n-1}}),$ 
which yields

$$b(S^nx)-b(x)= \langle v, \pi_{\btheta}(e_{x_0}\dots e_{x_{n-1}}) \rangle.$$
Since $v$ belongs to $E_{\btheta}^\perp$, we obtain
$$b(S^nx)-b(x)= \langle v, e_{x_0\dots x_{n-1}} \rangle=\phi_v(x, n).$$
The case $n <0$ is  handled in  the same way.

The continuity of $b$ is a consequence of the continuity of $\eta$.

\end{proof}

\subsection{On the rational  rank of asymptotic negligibles}

It turns out that the   asymptotically negligible cochains $\phi_v$ from \cref{subsec:an2} are generators of  $H^1_{\tt an}(X_{\btheta},\mathbb R)$.

\begin{theorem}\cite[Theorem 1.3]{Kell.Sad.17}\label{dim-asymptot-neglig}
Let $\btheta \in  {\mathbb R}^{d+1}_+$ be an irrational  direction.
Let $(X_{\btheta},S)$ be the hypercubic billiard  subshift with direction $\btheta$.  Let $L_{\btheta}^\perp$ stand for the hyperplane with normal vector $\btheta$. The  real vector space of asymptotically negligible cochains $H^1_{\tt an}(X_{\btheta},\mathbb R)$  has  dimension $d$ and is generated by the  classes of the functions $\phi_v$, for $v \in L_{\btheta}^\perp$ (as defined in \cref{subsec:an2}).
\end{theorem}
We recall that  equivalence classes are taken up to coboundaries of locally  constant  functions.
We complete this statement   by providing  explicitly a  set of generators
of $H^1_{\tt an}(X_{\btheta},\mathbb R)$  below.  This will be used in the discussion of \cref{subsec:back}.

\begin{proposition}\label{thm-cohom-an}
Let ${\btheta} \in {\mathbb R}^{d+1}_+$ be an irrational vector and 
let $(X_{\btheta},S)$ be  the associated  hypercubic billiard  subshift. 

The  real  vector  subspace  of $H^1_{\tt strong}(X_{\btheta},\mathbb R)$ generated by the $\psi_a$ ($a \in \mathcal A$)  coincides with the 
 real   vector subspace generated by the $\phi_v$ ($v \in L_{\btheta}^\perp$) and $\mathbf 1$.  It has dimension $d+1$.
 
Moreover,  the  $d$-dimensional real vector space of asymptotically negligible cochains $\cH^1_{\tt an}(X_{\btheta}, {\mathbb R})$ is generated by  the classes of the  functions  ${\mathbf 1}_a - \mu[a] {\mathbf 1}$,  for $a  \in \mathcal A$.
\end{proposition}

\begin{proof} 
 Note first  that the   characteristic  functions $ {\bf 1}_w$, for $w \in   \mathcal A^*$, are locally constant.
Their  classes   are   elements   of  $\cH^1_{\text{strong}}(X_{\btheta},\mathbb R)$.
Let $v \in  E_{\btheta}^\perp$. The cocycle $\phi_v$  defined in \cref{subsec:an2}
 is an asymptotic negligible cocyle by Lemma \ref{lem:phiv}.

$\bullet$
 We now prove that   the   (real) vector subspace of $\cH^1_{\tt{strong}}(X_{\btheta},\mathbb R)$  spanned by the   classes
of  characteristic  functions ${\bf 1}_a$, for the letters $a \in {\mathcal A}$, is equal to the subspace of $\cH^1_{\tt{strong}}(X_{\btheta},\mathbb R)$  spanned by the   classes
of $\psi_v$ ($v \in L_{\btheta}^\perp$) and ${\mathbf 1}$.

For any  given  letter  $a$  in $ \mathcal A$,
we define   $v_a$ in $L_{\btheta}^\perp$  as  the  vector uniquely defined by   the following system of $d$ equations
  \[
  \langle \pi_{\btheta}(e_b), v_a \rangle = 1, \quad b \in \mathcal A \setminus \{a\}.
 \]
 This system of  $d$ equations has indeed a unique solution since  $L_{\btheta}^\perp$  has dimension $d$ together with the
   irrationality of $\btheta$.
Indeed the  additive group generated by the  vector $f_i$, i.e.,  $\langle \mathbb Z  f_i \,:\, i=1 ,\ldots, d+1 \rangle$ is dense in $L_{\btheta}^\perp$. It implies that any subfamily of $d$ vectors among the set of vectors  $\{ f_i \,:\, i=1 ,\ldots, d+1 \}$ has to be independent,
and  the system above has a unique solution.
 We  then  set $\alpha_a
  =  \langle \pi_{\btheta}(e_a), v_a \rangle . $ Since $v_a \in L_{\btheta}^\perp$,  note that 
one has   $$\langle e_b, v_a \rangle = 1 \mbox{  for } b \neq a, \mbox{  and  }\langle e_a, v_a \rangle =\alpha_a. $$

Let us  first  prove that $ \alpha_a <0$ for all $a$.  One has, for all $n \in {\mathbb N}$,  \begin{align} \label{eq:psi}
  \phi_{v_a}(x, n) & = \sum_{i=0}^{n-1} \langle v_a, e_{x_i} \rangle  = \alpha_a |x_0 \cdots x_{n-1}|_a + 1 \bigl( n-|x_0 \cdots x_{n-1}|_a \bigr)  \notag \\
      & = (\alpha_{a} - 1) \psi_a (x, n) + {\mathbf l} (x, n).
 \end{align}

 Now for all $n$, $\pi_{\btheta}(e_{x_1 \cdots x_n}) \in W_{\btheta}$ (see the notation of  \cref{subsec:hypercube}), and in particular, it  is bounded (with respect to $n$).
 Therefore, $\phi_{v_a}(x,n)$ is bounded uniformly in $n$.
 It implies (by unique ergodicity) that 
 $(\alpha_a - 1)\mu[a]=-1$, which  implies that  $\alpha_a-1= -\frac{1}{\mu[a]}<-1$, and thus  $ \alpha_a <0 .$ 
  
We now can deduce   that any set of $d$ vectors among the $v_a, a\in\mathcal A$, is a set of $d$ linearly independent vectors.
 Indeed, assume, up to permutation and  by contradiction, that   there is a linear combination of the $d-1$ first vectors such that
 $$\sum_{ 1 \leq i \leq d-1}  \lambda_i v_i=v_{d}.$$
 By construction,  one has $\langle v_i, e_{d+1} \rangle = 1$ for $  1\leq i \leq d$. Therefore, by considering the scalar product with $e_{d+1}$, one obtains $\sum_{ 1 \leq i \leq d-1} \lambda_i=1$.
 Besides, $\langle v_i, e_{d} \rangle = 1$ for $i \leq d-1$ and  $ \alpha_{d}=\langle v_{d}, e_{d} \rangle  < 0$,   as proved in the previous paragraph.   Therefore,
  by considering the scalar product with $e_{d}$, one obtains
 $\sum_{ 1 \leq i \leq d-1} \lambda_i = \alpha_{d } <1  $, which  provides the  desired  contradiction.
 
We thus deduce that   the  $d$-dimensional space $ L_{\btheta}^\perp $ is generated by $v_1, \cdots ,v_{d}$.
 Since the map $v \mapsto \phi_v$ is a linear map, this implies that any $\phi_v$, with  $v \in L_{\btheta}^\perp$, is a linear combination of  the $\phi_{v_i}$'s for $i=1, \cdots, d$.
Hence,  we have proved that the subspace of $H^1_{\tt{strong}}(X_{\btheta},\mathbb R)$ spanned by the   classes
of $\phi_v$ ($v \in L_{\btheta}^\perp$) is generated by  $\phi_{v_1}, \cdots ,\phi_{v_{d-1}}$. 
 Moreover, we deduce  from  (\ref{eq:psi}) that, for all $a \in {\mathcal A}$,
 \begin{equation}\label{eq1}
 \phi_{v_a} =- (\mu[a])^{-1} \psi_a +  {\mathbf l}.
 \end{equation}
This  implies that  the real  vector space  generated by  the  $\psi_a$'s  ($a \in \{1, \cdots, d+1\}$)
coincides  with the real vector  space generated by the $\phi_v$'s,  for $v \in L_{\btheta}^{\perp}$, and  $\mathbf 1$.

\bigskip

$\bullet$  It remains  now to prove  that  the rank 
of   the subspace of $H^1_{\tt{strong}}(X_{\btheta},\mathbb R)$ spanned by the   classes
of $\psi_v$ ($v \in L_{\btheta}^\perp$) and ${\mathbf 1}$ is equal to $d+1$.
 Assume that $ \sum_{ 1 \leq i \leq d+1}  \lambda_i  \phi_{v_i} +\lambda  {\mathbf 1}=0$. 
 Let $x \in X_{\btheta}$. For all  nonnegative integer $n$, one has 
$$\langle   \sum_{ 1 \leq i \leq d-1}  \lambda_i  {v_i}  ,  e_{x_0 \cdots e_{x_n}}  \rangle + n \lambda= 0.$$ Since  $\langle   \sum_{ 1 \leq i \leq d-1}  \lambda_i  {v_i}  ,  e_{x_0 \cdots e_{x_n}}  \rangle$ is bounded   with respect to $n$, this implies that $\lambda=0$.
Since the vectors $v_i$ belong to $L_{\btheta}^ \perp$,  one has  for every $x$ and $n$
$$\langle   \sum_{ 1 \leq i \leq d-1}  \lambda_i  {v_i}  ,  e_{x_0 \cdots e_{x_n}}  \rangle =\langle   \sum_{ 1 \leq i \leq d-1}  \lambda_i  {v_i}  ,  \pi _{\btheta} (e_{x_0 \cdots e_{x_n}})  \rangle = 0.$$ By denseness of $\pi _{\btheta}( {\mathbb Z^{d+1}})$, one  deduces that 
 $ \sum_  { 1 \leq i \leq d+1}  \lambda_i  {v_i}=0$.  
By linear independence of  the  vectors  $v_i$,  $i=1, \cdots, d$,   we conclude that  $\lambda_i=0$ for all $i$.

$\bullet$  It remains to   prove that $\cH^1_{\tt an}(X_{\btheta},\mathbb R)$ is  generated  by the  classes of 
 functions 
 ${\mathbf 1}_a - \mu[a]{\mathbf 1}$,  $a \in {\mathcal A}$.
Note first that  $\cH^1_{\tt an}(X_{\btheta},\mathbb R)$ contains   the classes of the  functions 
${\mathbf 1}_a - \mu[a]{\mathbf 1}$, for  $a \in {\mathcal A}$. 
Indeed, by  \cref{thm-rappel}, one has balance on letters, and thus  ${\mathbf 1}_a  -  \mu[a]{\mathbf 1}$
is  an  asymptotically negligible function  by  \cref{cor-balanced-asympt}.
We conclude  by    a dimension argument  by  \cref{dim-asymptot-neglig}, since the dimension  of  $\cH^1_{\tt an}(X_{\btheta},\mathbb R)$is equal to $d$.

   \end{proof}

 \begin{remark} \label{rem:bal}
We deduce from Proposition \ref{thm-cohom-an} that if  $X_{\btheta} $ is balanced on $w$, then $\psi_w$ is cohomologous to a   linear   combination of the cocyles  $\psi_a$ and $\mathbf 1$.
We  also can deduce it directly.
Indeed, let $w$  be such that   ${\mathbf 1}_w - \mu[w]$ is  a coboundary of a continuous function.
Then,   $\mu[w]$ is an additive  eigenvalue, that is,
 $\mu[w]$ belongs to the additive  group  generated  by the $\mu[a]$, by Proposition \ref{prop-vp}. 
 Let $\lambda_a$,  for $ a \in  {\mathcal A}$,  be integers  such  that
 $\mu[w]= \sum \lambda_a \mu[a]$.
 One has
 $$   {\mathbf 1}_w - \sum  \lambda_a  {\mathbf 1}_a={\mathbf 1}_w - \mu[w] -  \sum \lambda_a ({\mathbf 1}_a- \mu[a]).$$
 Hence 
 $ {\mathbf 1}_w - \sum  \lambda_a  {\mathbf 1}_a$ is  a   coboundary of a  continuous  function,   as  a linear combination  of   coboundaries.
 Since it  takes integer values,  it is a coboundary of  a locally constant  function.
 Indeed  an  integer valued continuous function that is a real coboundary is a coboundary of a  locally constant  function  (\cite[Proposition 4.1]{Ormes:00}).
 One thus has $   {\mathbf 1}_w - \sum  \lambda_a  {\mathbf 1}_a\equiv 0$.
 This implies that   the class of  ${\mathbf 1}_w - \mu[w] $  coincides with the class  of $   \sum \lambda_a ({\mathbf 1}_a- \mu[a])$. It thus  belongs to  the 
  ${\mathbb R}$-vector space generated  by the 
 functions 
 ${\mathbf 1}_a - \mu[a]$, for $a $  in ${\mathcal A}$.

\end{remark}



\section{Proofs of main results}
We now have collected all the required material in order to prove  our main results.
\subsection{Proof of   \cref{theo:main}} \label{subsec:proofmain}

 By Proposition \ref{prop-jul}, $H^1_{\tt{strong}}(X_{\btheta},\mathbb R)$ is not finitely  generated as a ${\mathbb R}$-vector space if $d\geq 2$.  But  $H^1_{\tt{an}} (X_{\btheta},\mathbb R)$  is finitely  generated as a ${\mathbb R}$-vector space  by Theorem \ref{dim-asymptot-neglig}. Thus these  two sets are different.  Hence there exists 
a finite word $w$ such  that  the class  of $\psi_w  $  does not belong to   $H^1_{\tt{an}} (X_{\btheta},\mathbb R)$  (modulo  a coboundary of a  locally constant function). By Corollary \ref{cor-balanced-asympt} we conclude that the word  $w$ is not balanced  in $X_{\btheta}$.

Observe that the result does not  hold in the case $d=1$, which corresponds to the case of  Sturmian words which are balanced  on all their factors. In this case,  the group $H^1_{\tt{an}}(X_{\btheta},\mathbb R)$ has rank  $1$ over  ${\mathbb Q}$. More precisely we have $H^1_{\tt{an}}(X_{\btheta},\mathbb R)= \btheta {\mathbb Q}$, see \cite{For.Hunt.Kellen.02}.

\subsection{Proof of  \cref{prop:long2}} \label{subsec:proof2}

We follow the notation from \cref{subsec:cubic}.
We consider an  irrational   direction $\btheta=(1,\theta_1,\theta_2)\in \mathbb R^3_+$ and  the translation defined  on $\T^2$ by $\balpha=(\frac{\theta_1}{1+\theta_1+\theta_2},\frac{\theta_2}{1+\theta_1+\theta_2})$ that is  measurably  isomorphic  to the exchange of pieces   defined with respect to  the  partition  of $W_{\btheta}= \pi_{\btheta} ([0,1]^3)$ into  the three rhombi $W_{\btheta}^{(i)}$ for $ i=1,2,3$.  
We also consider the partition $\mathcal P_{\btheta}= \{W_{\btheta}^{(i)}\,:\, i=1,2,3\}$ of $W_{\btheta}$  (see  \cref{fig-tore}), and the partition $\mathcal P _{\btheta} \vee E^{-1}_{\btheta}\mathcal P_{\btheta}$, illustrated in \cref{fig-dim2}.  Each element of this partition corresponds to a factor of length two in the language of the billard shift $(X_{\btheta},S)$, and there are 7  factors of length $2$ (as recalled  in Theorem \ref{thm-rappel}; see also  \cite{Bary.95, Bed:09}).
We now apply the results  from \cite{GL:15} recalled in \cref{subsec:BRS}. Note that  measurable isomorphism  preserves bounded remainder sets for minimal  toral translations. 

We assume   $\theta_1>\theta_2>0$ and $\theta_1>1$.  The other cases are handled analogously. Note that the letter $2$ is thus the more frequent one, with 
 $\mu[1]=\frac{1}{1+\theta_1+\theta_2}$, $\mu[2]=\frac{\theta_1}{1+\theta_1+\theta_2}$ and $\mu[3]=\frac{\theta_2}{1+\theta_1+\theta_2}$. 
By  \cref{cor:GL}, the polygons which do not have a center of symmetry  are not bounded remainder sets, they thus do not correspond to balanced words. 
There are 6  of them. 

It remains to consider  the case of  the parallelogram with vertices
 $KBCD$  which corresponds to the factor $22$ (see \cref{fig-dim2}). 
Here we recall that  the lattice $\Lambda_{\btheta}$ is generated by the vectors  $f_3-f_1, f_2-f_1$ and that  the translation vector  is equal to $f_1$ modulo  the lattice $\Lambda_{\btheta}$, with 
  $$f_1=\frac{\theta_1}{1+\theta_1+\theta_2}(f_1-f_2)+\frac{\theta_2}{1+\theta_1+\theta_2}(f_1-f_3).$$

An easy computation gives the coordinates of the vertices   of   $KBCD$ in the hyperplane $E_{\btheta}^\perp$ with the help of Figure \ref{fig-dim2}:
$$\overrightarrow{OA}=f_1, \overrightarrow{OB}=f_1+\frac{\theta_2}{\theta_1} f_3,  \overrightarrow{OC}=f_1+f_3,  \overrightarrow{OD}=f_3+\frac{1}{\theta_1} f_1, \overrightarrow{ OK}=-f_2.  $$ 
Indeed,  let $\lambda>0$ be such that $\overrightarrow{KB}=\lambda f_1$ and $\mu>0$ such that $ \overrightarrow{AB}=\mu f_3$.
One has $$-f_2+\lambda f_1-\mu f_3-f_1=\vec {0},$$ by considering the cycle  going through the vertices $OKBA$.
By  $f_1+\theta_1f_2+\theta_2 f_3=0$   together with the  irrationality of $\btheta$, we deduce
that $\lambda=1-\frac{1}{\theta_1} $ and $\mu=\frac{\theta_2}{\theta_1}$.

Let us  apply Corollary \ref{cor:GL} to the parallelogram with vertices
 $KBCD$.  It is clear that $C, K$ are in $\mathbb Z f_1+\Lambda_{\btheta}$. 
Let us  prove that the vertices $B, D$ are not  in $\mathbb Z f_1+\Lambda_{\btheta}$. Indeed, if  $B$ belongs to $\mathbb Z f_1+\Lambda_{\theta}$, then there exist some integers $n,m,p$ such that 
$$f_1+\frac{\theta_2}{\theta_1}f_3=n f_1+m(f_1-f_3)+p(f_1-f_2).
 $$
 This gives
$(n+p-1+m)f_1-pf_2-(m+\frac{\theta_2}{\theta_1})f_3=0$. We deduce from the irrationality of $\btheta$  together with  the linear relation  $f_1+\theta_1 f_2+\theta_2 f_3=0$  that
$$n+p+m-1=-\frac{p}{\theta_1}=-\frac{m+\frac{\theta_2}{\theta_1}}{\theta_2},$$  which contradicts the irrationality of $\btheta$. A similar proof holds for the  point $D$.

 Consider  now the second condition  of Corollary \ref{cor:GL} for the pair of parallel edges   of  the parallelogram with vertices
 $KBCD$ provided  by  $\overrightarrow{KB }$ and $\overrightarrow{CD}$.
The vector between the midpoints of these two edges  cannot  belong to   the lattice $\mathbb Zf_1+\Lambda_{\btheta}$  since  the vectors $\overrightarrow{KB},\overrightarrow{CD}$  do not belong to this lattice, as proved above.  Also, the two edges under consideration  do  belong   to the lattice. Hence the second condition of  Corollary \ref{cor:GL} is violated, which implies that  $KBCD$ is not a bounded remainder set.

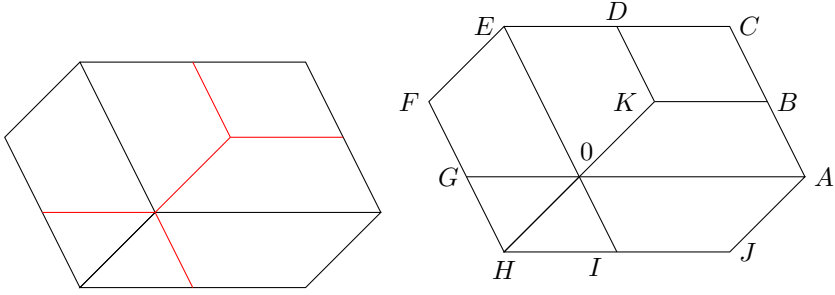
\begin{figure}
\begin{tikzpicture}

\draw (0,0)--(3,0)--++(1,1)--++(-3,0)--cycle;
\draw (0,0)--(-1,2)--++(1,1)--++(1,-2)--cycle;
\draw (0,3)--++(3,0)--(4,1);

\draw[red] (1,1)--(2,2);
\draw[red] (2,2)--(3.5,2);
\draw[red] (2,2)--(1.5,3);
\draw[red] (1,1)--(1.5,0);
\draw[red] (1,1)--(-0.5,1);
\end{tikzpicture}
\begin{tikzpicture}
\draw (3,0) node[right]{$J$};
\draw (1.1,1.1) node[right,above]{$0$};
\draw (-1,2) node[left]{$F$};
\draw (0,0) node[right,below]{$H$};
\draw (0,3) node[left]{$E$};
\draw (3,3) node[right]{$C$};
\draw (4,1) node[right]{$A$};
\draw (1.9,2) node[left]{$K$};
\draw (-1,1) node[right]{$G$};
\draw (1,-0.2) node[right]{$I$};
\draw (3.5,2) node[right]{$B$};
\draw (1.5,3.2) node{$D$};

\draw (0,0)--(3,0)--++(1,1)--++(-3,0)--cycle;
\draw (0,0)--(-1,2)--++(1,1)--++(1,-2)--cycle;
\draw (0,3)--++(3,0)--(4,1);

\draw (1,1)--(2,2);
\draw (2,2)--(3.5,2);
\draw (2,2)--(1.5,3);
\draw (1,1)--(1.5,0);
\draw (1,1)--(-0.5,1);
\end{tikzpicture}
\caption{The partition${\mathcal P}_{\btheta}\vee E_{\btheta}^{-1} {\mathcal P}_{\btheta}$ associated with two-letter  factors  of cubic billiard 
 words. There are $7$ polygons, each corresponding to a factor  of length  2 of the language.}\label{fig-dim2}
\end{figure}

We conclude  that  no  factor of length larger than or equal to $2$  is balanced,    by recalling that   balance on factors of length $n + 1$ implies   balance on factors of length $n$, 
 by \cite[Lemma 23]{Adam:03}.

\subsection{Back to  eigenvalues}\label{subsec:back}

 By Proposition \ref{prop-vp},  if a factor  $w$ is balanced, then $\mu[w]$ is an additive  eigenvalue.
 The aim of this section is to highlight cases where   the converse  does not hold. This explains  the need   of the  geometric statement of 
 Corollary  \ref{cor:GL} in the previous proof via the   close examination of the geometry
 of cells associated with two-letter factors  when applying 
 the geometric statement of 
 Corollary  \ref{cor:GL} in the previous proof.  Indeed, 
there exist   minimal  billiard subshifts $(X_{\btheta},S)$  in $\mathbb R^3$  that admit a   factor  $w$ in the language  ${\mathcal L}(X_{\btheta})$ such that the
subshift  $X_{\btheta}$ is not  balanced on  $w$, with
yet  $\mu[w]$  being an  additive eigenvalue of the subshift. According to \cref{thm-cohom-an},  this shows that even  if
$\mu[w]$ belongs to  the   additive group $\langle \mu[i] \,:\, i=1, \ldots, d+1\rangle$,
then  ${\mathbf 1}_w - \mu[w] {\mathbf 1}$  is not a linear combination of  the
 ${\mathbf 1}_i - \mu[i] {\mathbf 1}$, $i=1, \ldots, d+1$.

We assume   $\theta_1>\theta_2$, $\theta_1>1$, as previously.
We  consider the partition  $\mathcal P_{\btheta} \vee T^{-1}\mathcal P_{\btheta} $ discussed in \cref{subsec:proof2}. Each element of this partition corresponds to a word of length two in the language and  the frequency of the word is proportional to the measure  of the  corresponding element of the partition. One checks  more precisely that
$$
\begin{cases}
\mu[1]=\frac{1}{1+\theta_1+\theta_2},\\
\mu[2]=\frac{\theta_1}{1+\theta_1+\theta_2},\\
\mu[3]=\frac{\theta_2}{1+\theta_1+\theta_2},\\
\mu[31]=\mu[13]= \frac{\theta_2}{2\theta_1(1+\theta_1+\theta_2)}=\frac{1}{2\theta_1}\mu[3], \\
\mu [21]= \mu[12]= \frac{2\theta_1-\theta_2}{2\theta_1(1+\theta_1+\theta_2)}= \frac{(2\theta_1-1)}{2\theta_1}\mu[1],\\
\mu[23]=\mu[32]=\frac{v(2\theta_1-1)}{2\theta_1(1+\theta_1+\theta_2)}= \frac{(2\theta_1-1)}{2\theta_1}\mu[3]\\
 \mu[22]
 =\frac{ (\theta_1-1)(\theta_1-\theta_2)}{\theta_1(1+\theta_1+\theta_2)} 
 =\frac{(\theta_1-1)(\theta_1-\theta_2)}{\theta_1}\mu[1].
 \end{cases}
$$

The measures of factors of length $2$ belong to $\frac{1}{2\theta_1}(\mathbb Z \mu[1]+   \mathbb Z \mu[2]+   \mathbb Z \mu[3])$, whereas   the   group of additive eigenvalues  is $\mathbb Z \mu[1]+   \mathbb Z \mu[2]+   \mathbb Z \mu[3] $.
Let us assume that $ \frac{1}{2\theta_1} \in  \mathbb Z+ \theta_1  \mathbb Z+ \theta_2  \mathbb Z$.
The  element of the two-letter partition  $\mathcal P_{\btheta} \vee T^{-1}\mathcal P_{\btheta} $ associated with the factor   $31$ is not a bounded remainder set (since being a triangle), even if  its measure  belongs to the 
group of additive eigenvalues.

\bibliographystyle{alpha}
\nocite{*}
\bibliography{biblio-eqbillard}

\newpage

\end{document}